\documentclass[12pt,reqno]{amsart}

\usepackage{amsmath,amssymb,amsthm,mathtools,enumitem}
\usepackage[margin=1in]{geometry}
\usepackage{makecell}

\newcommand{\Des}{\mathrm{Des}}
\newcommand{\SGE}{S_{\ge 2}}

\usepackage{diagbox}
\usepackage{makecell}

\usepackage{amssymb,latexsym,cite,longtable,url,array,booktabs}
\usepackage[linguistics]{forest}
\usepackage{multicol}
\usepackage{arydshln,pstricks,pst-node,ytableau}
\usepackage{multicol,float,vcell}
\usepackage{lipsum}
\usepackage{comment}
\usepackage{tabularx,booktabs}
\usepackage{tikz}
\usetikzlibrary{positioning}
\usepackage{hyperref}

\newcolumntype{Z}[0]{>{\centering\arraybackslash}X}
\newcolumntype{n}[0]{>{\hsize=.8\hsize}Z}
\newcolumntype{s}[0]{>{\hsize=.5\hsize}Z}

\advance\textheight by \topskip  \makeatletter
\pgfkeys{/ytableau/options,
 onlyboxsize/.value required,
 onlyboxsize/.code = {%
  \compare@YT{#1}{normal}%
  \ifeq@YT
   \xdef\macro@boxdim@YT{\expandonce@YT\boxdim@normal@YT}%
  \else
   \xdef\macro@boxdim@YT{#1}%
  \fi
 }
}
\makeatletter

\newenvironment{smallytableau}{%
  \ytableausetup{smalltableaux,onlyboxsize=1em}%
  \begin{ytableau}%
}{%
  \end{ytableau}%
  \ytableausetup{nosmalltableaux,boxsize=normal}%
}

\usepackage{tikz}
\usepackage{xparse}

\makeatletter
\newcommand*{\rom}[1]{\expandafter\@slowromancap\romannumeral #1@}
\makeatother

\newcount\tableauRow
\newcount\tableauCol

\newenvironment{Tableau}[1]{%
  \tikzpicture[scale=0.7,draw/.append style={thick,black},
                      baseline=(current bounding box.center)]
    \tableauRow=-1.5
    \foreach \Row in {#1} {
       \tableauCol=0.6
       \foreach\k in \Row {
         \draw[thin](\the\tableauCol,\the\tableauRow)rectangle++(1,1);
         \draw[thin](\the\tableauCol,\the\tableauRow)+(0.5,0.5)node{$\k$};
         \global\advance\tableauCol by 1
       }
       \global\advance\tableauRow by -1
    }
}{\endtikzpicture}

\makeatletter
\renewcommand{\paragraph}{\@startsection{paragraph}{4}{0pt}%
   {-3.25ex plus -1ex minus -0.2ex}%
   {1.5ex plus 0.2ex}%
   {\normalfont\normalsize\bfseries}}
\makeatother

\stepcounter{secnumdepth}
\stepcounter{tocdepth}

\usepackage{microtype}
\usepackage{nicematrix}
\NiceMatrixOptions
{
 custom-line =
 {
letter = ; ,
command = \hdashedline ,
ccommand = \cdashedline ,
tikz = { dash pattern = on 4 pt off 2pt } ,
total-width = \pgflinewidth
 }
}

\newcommand{\BBB}{{\mathcal{B}}}

\newtheorem{theorem}{Theorem}[section]
\newtheorem{definition}[theorem]{Definition}
\newtheorem{thm}[theorem]{Theorem}
\newtheorem{prop}[theorem]{Proposition}
\newtheorem{example}[theorem]{Example}
\newtheorem{cor}[theorem]{Corollary}
\newtheorem{claim}[theorem]{Claim}
\newtheorem{lemma}[theorem]{Lemma}
\newtheorem{remark}[theorem]{Remark}

\newtheorem{question}[theorem]{Question}

\newcommand{\Q}{{\mathcal Q}}
\newcommand{\F}{\mathcal{F}}

\title{On the Schur-positivity of various sets of set partitions}
\author{Eli Bagno and David Garber}

\address{Eli Bagno, Jerusalem College of Technology, 21 HaVaad HaLeumi St., Jerusalem, Israel and Michlalah College Jerusalem, 36 Barukh Duvdevani St., Jerusalem, Israel, \texttt{bagnoe@g.jct.ac.il}}
\address{David Garber, Holon Institute of Technology, 52 Golomb St., P.O.Box 305, 5810201 Holon, Israel,  \texttt{garber@hit.ac.il}}

\begin{document}

\maketitle

\begin{abstract}
A symmetric function is called {\it Schur-positive} if it admits an expansion in the Schur basis with nonnegative coefficients. In this paper, we study the Schur positivity of symmetric functions naturally associated with set partitions, with respect to two different notions of descent.

In the first case, the Schur expansion involves hook-shaped Young diagrams, and the corresponding coefficients are given by Touchard-Riordan polynomials, which enumerate matchings by their number of crossings. In the second case, the Schur functions correspond to two-rows Young diagrams, and the coefficients are partial sums of associated Bell numbers.

A key ingredient of our approach in the second case is the notion of a removable singleton, defined algebraically and shown to admit an equivalent combinatorial interpretation via jeu-de-taquin rectification of skew tableaux.

As an application, we establish Schur positivity for various classes of symmetric functions indexed by non-crossing partitions and partitions with a given number of parts. We provide an explicit combinatorial description of the tableaux that contribute to the Schur expansion, and we connects the obtained coefficients to some known integer sequences.
\end{abstract}

\section{Introduction}

The study of Schur-positivity for combinatorial families is a central topic in algebraic combinatorics, with a particularly well-developed theory in the context of permutations. Foundational work of Gessel \cite{Gessel1984} introduced quasi-symmetric functions as generating functions for descent statistics, while Stanley \cite{EC2} established their deep connections with symmetric functions and representation theory. In this framework, many permutation classes give rise to symmetric and Schur-positive generating functions, a phenomenon that is now well understood through tools such as dual equivalence, introduced by Assaf \cite{Assaf}, which provides a general combinatorial mechanism for proving Schur-positivity.

A significant work on Schur-positivity for permutation sets has been developed by Athanasiadis, Adin and Roichman \cite{AAR}, who introduced systematic methods for constructing and analyzing Schur-positive families via descent statistics and Coxeter-theoretic structures.
Elizalde and Roichman \cite{ER} introduced general constructions for studying Schur-positivity of permutations via grid classes and product operations. Subsequent developments can be found e.g. in Bloom et al. \cite{BER}.

\medskip

In contrast, the extension of these ideas to set partitions is far less developed.
For the special case of matchings, which can be viewed as set partitions with block size at most two,  several independent approaches have established Schur-positivity for natural families. Marmor \cite{Marmor} introduced a Knuth-type equivalence on matchings and proved Schur-positivity for families defined by crossing statistics and related constraints.  Adin and Roichman \cite{AR} and Yan et al.\cite{YYZ} studied descent statistics on involutions and matchings, showing that the associated quasi-symmetric generating functions are symmetric and Schur-positive. These results indicate that key structural features from the permutation setting - such as descent statistics and Knuth-like relations - extend, at least partially, to matchings.

By encoding set partitions using restricted-growth functions (see e.g. the book of E\u{g}ecio\u{g}lu and  Garsia \cite[Sec. 1.7]{EgGa}) and defining descent-like statistics (see, e.g., Wachs-White \cite{WW}), one can define quasi-symmetric generating functions for set partitions  \cite{BGMS}.
However, there is as yet no analogue of the Robinson-Schensted correspondence or dual equivalence theory that would systematically explain Schur-positivity. Consequently, existing results are largely confined to specific subclasses or constructions, and the development of a unified theory of Schur-positivity for set partitions remains an open direction of research.

\medskip

In this paper, we deal with two descent functions over set partitions, which are motivated by results of Adin and Roichman \cite{AR}. They dealt with the descent-type parameter on matchings which they called {\it geometric descent}: an index $i \in [n-1]:=\{1,2,\dots,n-1\}$ is a {\it geometric descent} of a matching $m$ on $[n]$ if
one of the following conditions holds:
\begin{enumerate}
\item  $\{i, i + 1\}$ is a matched pair in $m$,
\item The arc containing $i$ intersects the arc containing $i + 1$,
\item $i$ is unmatched and $i + 1$ is matched.
\end{enumerate}

Hence, we define two different types of descent parameters on set partitions. The first one, denoted ${\rm Short}(\pi)$, contains the elements $i$, such that $i$ and $i+1$ share the same block in a set partition $\pi$, while the second, denoted ${\rm Desing}(\pi)$, contains the elements $i$, such that $i$ is a singleton in $\pi$ but $i+1$ is not.

\medskip

We prove Schur-positivity of some subsets of the set $\mathcal{PS}et(n)$ of set partitions. Our main results are of the following general form, where $\mathcal{A}$ is a specific subset of $\mathcal{PS}et(n)$, ${\rm \mathcal{D}es}(\pi)$ is a specific choice of a descent set function, which is either ${\rm Short}$ or ${\rm Desing}$, $\mathcal{F}_{n,{\rm \mathcal{D}es}(\pi)}$ is the fundamental quasi-symmetric function with respect to that choice of ${\rm \mathcal{D}es}$, $\left\{\mathcal{C}_{n,k}\right\}$ is a sequence of coefficients and $\{\lambda_k\}$ is a sequence of shapes determining the Schur functions $s_{\lambda_k}$:
$$\sum_{\pi \in \mathcal{A}}\mathcal{F}_{n,{\rm \mathcal{D}es}(\pi)}=\sum\limits_{k}\mathcal{C}_{n,k} s_{\lambda_k}.$$

We list our main results in Table \ref{table main results}: the first column indicates the corresponding set, the second column indicates the chosen descent function, the third column mentions the corresponding coefficients and the fourth one is the shape of the Young tableaux, where the last column indicates the whereabouts of the corresponding result in this paper. The different coefficients appearing in the table are as follows:
\begin{itemize}
\item $T_{n,j}(q)$ are the Touchard-Riordan polynomials (see Section \ref{ss crossing number}),
\item $M_n$ are the Motzkin numbers (see Section \ref{subsec:motzkin numbers}),
\item $M_{n,k}$ are the elements of the Motzkin triangle (see Section \ref{subsec:motzkin numbers}),
\item  ${\rm Bell}_{\geq 2}(n)$ are the associated Bell numbers (see Definition \ref{def associated Bell numbers}),
\item $r_i$ are the Riordan numbers (see Definition \ref{definition Riordan numbers}),
\item $S_{\geq 2}(n,k)$ are the associated Stirling numbers of the second kind (see Definition \ref{definition associated Stirling number}).
\end{itemize}

\begin{table}[H]
\vspace{-140pt}
\begin{footnotesize}
\begin{tabular}{||c|c|c|c|c||}
&&&&\\
&&&&\\
&&&&\\
&&&&\\
&&&&\\
&&&&\\
&&&&\\
&&&&\\
&&&&\\
&&&&\\
&&&&\\
\hline\hline
$\mathcal{A}$ & ${\rm \mathcal{D}es}(\pi)$ & $\mathcal{C}_{n,k}$  & $\lambda_k$ & Thm. \\
\hline\hline
 $\mathcal{PS}et(n,\ell)$  & {\rm  Short}   &  $T_{n-1-k,n-\ell-k}(1)$ & $(n-k,1^k)$ &
   \ref{specializations}(1)  \\
\hline
$\mathcal{PS}et(n)$  & {\rm  Short}  &  $\sum\limits_{\ell=1}^{n-k}T_{n-1-k,n-\ell-k}(1)$ & $(n-k,1^k)$ &
   \ref{specializations}(2)  \\
\hline
${\mathcal NC}(n,\ell)$  & {\rm  Short} &  $M_{n-1-k,n-\ell-k}$ & $(n-k,1^k)$ &
   \ref{specializations}(3)  \\
\hline
${\mathcal NC}(n)$  & {\rm  Short}  &  $M_{n-k-1}$ & $(n-k,1^k)$ &
   \ref{specializations}(4)  \\
\hline
$\mathcal{PS}et(n)$ & {\rm  Desing}  & $\left\{ \begin{array}{ll}
\sum\limits_{i=2}^{n-1} {\rm Bell}_{i\geq 2}(i) & k=1\\
\sum\limits_{i=k}^{n-k} {\rm Bell}_{i\geq 2}(i) & (k=0) \mbox{ or } (k>1). \\
\end{array}\right.$ &  $(n-k,k)$ &    \ref{main theorem}  \\
 \hline
${\mathcal NC}(n)$ & {\rm  Desing}  & $ \left\{ \begin{array}{ll}
\sum\limits_{i=2}^{n-1} r_i & k=1\\
\sum\limits_{i=k}^{n-k} r_i & (k=0) \mbox{ or } (k>1) \\
\end{array}\right.$  & $(n-k,k)$ & \ref{main theorem noncrossing}\\
\hline
$\mathcal{PS}et(n,b)$ & {\rm  Desing} & $\left\{ \begin{array}{ll}
\sum\limits_{t=1}^{n-2} S_{\geq 2}(n-t,b-t) & \ \ k=1\\
\sum\limits_{t=k}^{n-k} S_{\geq 2}(n-t,b-t) & \ \ (k=0) \mbox{ or } (k>1) \\
\end{array}\right.$  & $(n-k,k)$ & \ref{main theorem blocks}\\
\hline\hline
\end{tabular}
\end{footnotesize}
\caption{List of our main results}\label{table main results}
\end{table}

Our results, dealing with descent set function based on singletons, provide an application of the general theory of Marmor \cite{Marmor}, which states that for sets $\mathcal{A}$ with {\it sparse} descent sets $\rm \mathcal{D}es$ (where a set $J \subset [n-1]$ is called   {\em sparse} if $\{j, j + 1\} \not\subseteq J$ for every
$1 \leq j \leq n-2$; termed also {\em lacunar} by Grinberg \cite{Grinberg}), the quasi-symmetric function $\mathcal{Q}_{{\rm \mathcal{D}es},n}$ is positively-spanned by Schur functions corresponding to two-rows shapes. Explicitly, Marmor \cite{Marmor} proved (adapted to our notations):
\begin{thm}
Let $\mathcal{A}$ be a finite set with a set function $\mathcal{D}{\rm es}: \mathcal{A} \to 2^{[n-1]}$. Then the following statements are equivalent:
\begin{itemize}
\item $\mathcal{D}{\rm es}$ is sparse, and for every sparse $J\subseteq [n-1]$, the cardinality of the set
$$\{\pi \in \mathcal{A} \mid \mathcal{D}{\rm es}(\pi) \supseteq J\}$$ depends only on the size of $J$.

\medskip

\item ${\mathcal A}$ is symmetric with respect to $\mathcal{D}{\rm es}$, with a Schur expansion of the form
$\mathcal{Q}_{n,\mathcal{D}{\rm es}}(\mathcal{A}) = \sum\limits_{k=0}^{\left\lfloor\frac{n}{2}\right\rfloor} c_k s_{(n-k,k)}$ for some $c_k \in \mathbb{Z}$.
\end{itemize}

\medskip

Furthermore, if these statements hold, then $\mathcal{A}$ is Schur-positive, and its Schur expansion is $$\mathcal{Q}_{n,\mathcal{D}{\rm es}}(\mathcal{A})=\sum\limits_{k=0}^{\left\lfloor \frac{n}{2} \right\rfloor} \underbrace{\left|\{\pi \in  \mathcal{\mathcal{A}} \mid \mathcal{D}{\rm es}(\pi) = \{1, 3, 5, \dots, 2k-1\}\right|}_{=c_k} \cdot s_{(n-k,k)}.$$
\end{thm}

Our results improve Marmor's theorem in the cases we dealt with, as we provide the appropriate mappings from the sets $\mathcal{A}$ to the set of standard Young tableaux of the corresponding shapes and also supply combinatorial meanings to the coefficients taking part in the decompositions.

\medskip
This paper is organized as follows.
In Section \ref{Background}, we
give the requi background on Schur-positivity and the crossing number of a set partition.
In Section \ref{section on schur positivity for des sharing block},
we show that $\mathcal{PS}et(n)$, with the parameter  ${\rm Short}$ functioning as the descent function, is Schur-positive. Moreover, we derive similar results for some of its subsets (i.e. set partitions with a given number of blocks, and the non-crossing set partitions).
In Section \ref{subsection schur positivity with respet to desing}, we show that $\mathcal{PS}et(n)$, with the parameter  ${\rm Desing}$ functioning as the descent function, is Schur-positive.
In Section \ref{section: desing for noncrossing partitions},
we prove the same for non-crossing partitions of $[n]$. In Section \ref{section:refinement}, we refine our theory of Schur positivity with respect to the Desing parameter for partitions with a given number of blocks. In this section, we obtained some interesting connections to various integer sequences, as the numbers of total partitions, the Eulerian numbers and the second-order Eulerian numbers.
In Section \ref{section: distribution Desing parameter}, we change our course and provide an enumerative account of the parameter Desing.
Section \ref{section: open questions} concludes the paper with some open questions.

\section{Background}\label{Background}

In this section, we supply the required background  for  Schur-positivity in Section \ref{sec:prel_quasi}, the crossing number of a set partition and the Touchard-Riordan polynomials in Section \ref{ss crossing number} and introduce the Motzkin numbers and Motzkin triangle in Section
\ref{subsec:motzkin numbers}.

\subsection{Symmetric functions, quasi-symmetric functions and Schur positivity}\label{sec:prel_quasi}

In this subsection, we follow \cite{ABR}.
Let ${\bf x}=\{x_1,x_2,\dots\}$ be a countably infinite set of commuting variables and consider the algebra of formal power series over $\mathbb{Q}$.  

A power series $f\in \mathbb{Q}[[\bf{x}]]$ is called {\em symmetric} if it has a bounded degree and it is invariant under permutation of variables.
Denote by ${\rm Sym}_n$ the vector space of symmetric functions, homogeneous of degree $n$.  Bases for ${\rm Sym}_n$  are indexed by {\em partitions $\lambda$ of $n$} (denoted $\lambda \vdash n$), or equivalently, by Young diagrams with $n$ boxes.

The celebrated Schur basis of ${\rm Sym}_n$ is essential for our needs.
Given a partition $\lambda$ of $n$, a {\em standard Young tableau (SYT) of shape $\lambda$} is obtained by filling the boxes of the corresponding Young diagram bijectively with the elements $\{1,\dots,n\}$ so that rows and columns increase.  In a {\em semistandard Young tableau (SSYT) of shape $\lambda$}, we demand that the entries are positive integers such that rows weakly increase and columns strictly increase.
We write ${\rm SYT}(\lambda)$ or ${\rm SSYT}(\lambda)$ for the set of standard or semistandard Young tableaux of shape $\lambda$, respectively.

Next, we define the notion of a {\it descent set} of a standard Young tableau $T$:
\begin{definition}\label{Des of T}
The  {\em descent set of a standard Young tableau $T$} is defined to be:
$$
\Des( T) =\{i \mid \text {$i+1$ is in a lower row in $T$ than $i$}\}.
$$
\end{definition}
For example, if
$T=\raisebox{4mm}{\begin{smallytableau} 1&3&4 \\ 2&5&6 \\ 7 \end{smallytableau}}\ ,$
then $\Des (T)=\{1,4,6\}$.

\medskip
In this work, a special attention will be given  to standard Young tableaux having two rows, i.e., of shape $(n-k,k)$. Note that the first row of such a tableau comprises a sequence of runs of consecutive increasing numbers, where each run, apart from the last one, ends with a descent.

The {\it Schur function corresponding to a partition $\lambda \vdash n$} is: $s_{\lambda}=\sum\limits_{T \in {\rm SSYT}(\lambda)}\prod\limits_{i \in T}x_i.$
For example, the set ${\rm SSYT}(3,1)$ contains inter alia the following semistandard fillings of the shape
$\lambda=
{\begin{smallytableau} {}&{}&{} \\ {} \end{smallytableau}}
$\ :
$$
T_1 = {\begin{ytableau} {1}&{1}&{1} \\ {2} \end{ytableau}}\ , \ \
T_2= {\begin{ytableau} {1}&{2}&{2} \\ {2} \end{ytableau}}
\ , \ \  T_3={\begin{ytableau} {1}&{2}&{3} \\ {4} \end{ytableau}},
$$
which contribute to $s_{(3,1)}$ the monomials $m_1=x_1^3x_2,\ m_2=x_1x_2^3$ and $m_3=x_1x_2x_3x_4$ respectively.
It is well-known \cite[Chap. 7]{EC2} that the set $\{s_{\lambda}\}_{\lambda \vdash n}$ is a basis of ${\rm Sym}_n$.

A symmetric function is called {\em Schur-positive} if all the
coefficients in its expansion in the basis of Schur functions are nonnegative.
Determining whether a given symmetric function is
Schur-positive is a major problem in contemporary algebraic
combinatorics~\cite{Stanley_problems}.

\medskip

Quasi-symmetric functions extend symmetric functions. Here is the formal definition:
\begin{definition}\label{def quasi-symmetric}
A {\em quasi-symmetric function} is  a formal power series $g\in \mathbb{Q}[[{\mathbf x}]]$ of bounded degree  satisfying that any two of its monomials $x_{i_1}^{n_1}\dots x_{i_k}^{n_k}$ (where $i_1<\dots<i_k$) and $x_{j_1}^{n_1}\dots x_{j_k}^{n_k}$ (where $j_1<\dots<j_k$) have the same coefficient in $g$.
\end{definition}
Clearly, every symmetric function is quasi-symmetric, but not conversely; for example, $\sum\limits_{i<j}{x_i^2 x_j}$  is quasi-symmetric but not symmetric.

The fundamental quasi-symmetric functions, which are defined now, serve as a basis for the vector space of homogeneous quasi-symmetric functions of degree $n$:
\begin{definition}\label{def fund}
For each subset $D \subseteq [n-1]$ define the {\em
fundamental quasi-symmetric function}
\[
\mathcal{F}_{n,D}({\bf x}) := \sum_{i_1\le i_2 \le \ldots \le i_n \atop {i_j <
i_{j+1} \text{ if } j \in D}} x_{i_1} x_{i_2} \cdots x_{i_n}.
\]
\end{definition}

Let $\BBB$ be a (multi-)set of combinatorial objects, equipped with a {\em descent function}\break ${\rm \mathcal{D}es}: \BBB \to P([n-1])$, which associates to each
element $b\in \BBB$ a subset ${\rm \mathcal{D}es}(b) \subseteq [n-1]$. Define the
quasi-symmetric function
\[
\Q_n(\BBB) := \sum\limits_{b\in \BBB} m(b,\BBB) \F_{n,{\rm \mathcal{D}es}(b)},
\]
where $m(b,\BBB)$ is the multiplicity of the element $b \in \BBB$.

\medskip

Gessel showed that (see \cite[Theorem 7.19.7]{EC2}):
\begin{thm}[Gessel]\label{gessel}
For every partition $\lambda \vdash n$,
$\Q_n({{\rm SYT}(\lambda)})=s_{\lambda}$.
\end{thm}

Hence, proving Schur-positivity of a set $E$ with respect to some descent function, amounts to defining a descent-preserving function $\varphi: E \to \bigcup\limits_{\lambda \vdash n}{\rm SYT}(\lambda)$, such that for each $\lambda \vdash n$ and $S,T \in {\rm SYT}(\lambda)$, one has $|\varphi^{-1}(S)|=|\varphi^{-1}(T)|$.

\medskip

For each $0\leq k\leq n$, recall the definition of the {\it hook diagram} $\lambda_k=(n-k,1^k)$ (where $1^k$ stands for $k$ parts of size $1$), and note that:
\begin{equation}
\{{\rm Des}(T) : T\in {\rm SYT}(\lambda_k)\}=\{D\subseteq [n-1]: |D|=k\}.
\end{equation}

\subsection{The crossing number of a set partition and the Touchard-Riordan polynomials}\label{ss crossing number}

Let $n$ be a positive integer and recall that $\mathcal{PS}et(n)$ is the set of all set partitions of $[n]$.
Let $\pi=\{ B_1, B_2 , \dots , B_k\} \in \mathcal{PS}et(n)$ be a set  partition of  $[n]$ into $k$ blocks $B_1,\dots,B_k$, and assume that the elements in each block $B_i$ are ordered increasingly.
The partition $\pi$ can be depicted graphically as a linear graph
made of the $n$ points
$\{1,\dots,n\}$, drawn on a (virtual) horizontal line, ordered increasingly by their labels. For every block $B_j=\{j_1<\dots<j_t\}$, a semicircular arc is drawn between every pair of points $j_k$ and $j_{k+1}$. In the case that $j_{k+1}=j_k+1$, we join $j_k$ and $j_{k+1}$ by a straight line instead of a semicircular arc.

\begin{example}\label{exam_partition_crossing}
Let $\pi=137\mid 24568$ be a set partition of the set $[8]$ (which is an abbreviation for the partition $\{\{1,3,7\},\{ 2,4,5,6,8\}\}$). Then its corresponding linear graph is drawn in Figure \ref{fig def crossing}.

\begin{figure}[H]
\begin{center}
\begin{tikzpicture}[scale=1]

  \foreach \x in {1,...,8} {
    \fill (\x,0) circle (2pt);
    \node[below] at (\x,0) {\x};
  }

\fill[gray] (2.32,0.67) circle (2pt);
\fill[gray] (3.5,0.77) circle (2pt);
\fill[gray] (6.59,0.71) circle (2pt);

  \def\arc#1#2#3{
    \draw (#1,0) arc[start angle=180, end angle=0, x radius={(#2-#1)/2}, y radius=#3];
  }

  \arc{1}{3}{0.7}
  \arc{3}{7}{1.2}
  \arc{2}{4}{0.9}
  \draw (4,0) -- (6,0);
  \arc{6}{8}{0.8}
\end{tikzpicture}
\caption{}\label{fig def crossing}
\end{center}
\end{figure}

\end{example}

The notion of the {\it crossing number} of a set partition admits several classical variants,
depending on whether one counts crossings or measures their maximal size, and on the precise definition of a crossing.
Using the linear graph of the set partition, a {\it crossing} is typically defined as follows (see e.g. \cite{PoYa}).
\begin{definition}
Let $\pi$ be a set partition of the set $[n]$, presented as a linear graph as above.
For $r<s$, two arcs $(i_r,j_r)$ and $(i_s,j_s)$ {\em cross} each other if $i_r<i_s<j_r<j_s$. We denote by ${\rm cr}(\pi)$ the number of such crossings and call it the {\em crossing number} of the set partition $\pi$.
\end{definition}
In Example \ref{exam_partition_crossing}, the crossings are depicted as gray points, so we have ${\rm cr}(\pi)=3$.

\begin{remark}
A different definition of the crossing number can be found in \cite{CDDSY} as the size of the largest $k$-crossing, i.e., a family of arcs $(i_1,j_1),\dots,(i_k,j_k)$
with $i_1<\cdots<i_k<j_1<\cdots<j_k$.
Additional variants, which arise when restricting crossings to arcs from distinct blocks or
when strengthening the notion to crossings of blocks rather than arcs, can be found e.g. in  \cite{CDDSY} as well.
\end{remark}

The crossing number can be similarly defined on a subset of $\mathcal{PS}et(n)$,  namely, the set of {\it matchings} of $[n]$, which we define now.

\begin{definition}
A partition $\pi\in \mathcal{PS}et(n)$ in which each block contains at most $2$ elements is called a {\em matching}. If $n$ is even and all the blocks are of size $2$, then $\pi$ is called a {\em perfect matching}.
A matching of the set $[n]$ can be depicted as a {\em chord diagram}, which is a circle containing $n$ points, labeled by the numbers $1,\dots, n$ ordered clockwise.  Each block of size $2$ is represented by a chord connecting its two elements, see Figure \ref{crossing perfect}. Note that any matching on $n$ points can be also considered as an involution of the set $[n]$.
\end{definition}

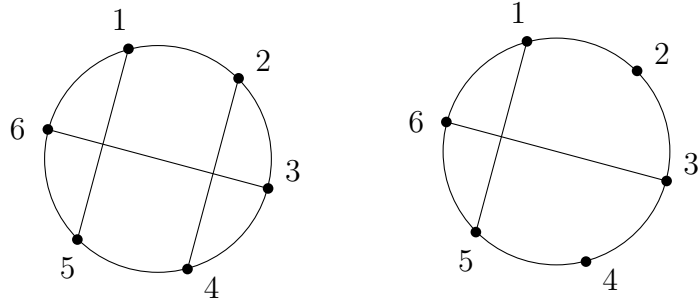
\begin{figure}[H]
    \centering
\hspace{-150pt}\begin{tikzpicture}
\node[draw=none,rotate=45, minimum size=3cm,regular polygon,regular polygon sides=6] (a) {};
  \fill (a.corner 1) circle[radius=2pt] node[shift={(72+35:0.4)}] {1};
  \fill (a.corner 2) circle[radius=2pt] node[shift={(144+35:0.4)}] {6};
  \fill (a.corner 3) circle[radius=2pt] node[shift={(216+35:0.4)}] {5};
  \fill (a.corner 4) circle[radius=2pt] node[shift={(288+35:0.4)}] {4};
  \fill (a.corner 5) circle[radius=2pt] node[shift={(35:0.4)}] {3};
    \fill (a.corner 6) circle[radius=2pt] node[shift={(35:0.4)}] {2};

\draw[black]     (a.corner 1) -- (a.corner 3);
\draw[black]     (a.corner 2) -- (a.corner 5);
\draw[black]     (a.corner 4) -- (a.corner 6);

\draw (0,0) circle (1.5cm);
\end{tikzpicture}

\vspace{-120pt}\hspace{150pt}\begin{tikzpicture}
\node[draw=none,rotate=45, minimum size=3cm,regular polygon,regular polygon sides=6] (a) {};
  \fill (a.corner 1) circle[radius=2pt] node[shift={(72+35:0.4)}] {1};
  \fill (a.corner 2) circle[radius=2pt] node[shift={(144+35:0.4)}] {6};
  \fill (a.corner 3) circle[radius=2pt] node[shift={(216+35:0.4)}] {5};
  \fill (a.corner 4) circle[radius=2pt] node[shift={(288+35:0.4)}] {4};
  \fill (a.corner 5) circle[radius=2pt] node[shift={(35:0.4)}] {3};
    \fill (a.corner 6) circle[radius=2pt] node[shift={(35:0.4)}] {2};

\draw[black]     (a.corner 1) -- (a.corner 3);
\draw[black]     (a.corner 2) -- (a.corner 5);

\draw (0,0) circle (1.5cm);
\end{tikzpicture}

\caption{Left: a perfect matching, right: a non-perfect matching}
\label{crossing perfect}
\end{figure}

\begin{definition}\label{definition of }
The {\em crossing number of a matching} is the crossing number of its corresponding partition as defined above, which can also be seen as the number of pairs of crossing chords in the chord diagram.
\end{definition}
The generating function of all perfect matchings on $2m$ points with respect to the crossing number is given by {\it Touchard-Riordan polynomials}, which are defined as follows, see \cite{Riordan75}:
$$\sum\limits_{M}q^{{\rm cr}(M)}=T_{2m}(q):=\frac{1}{(1-q)^m}\sum\limits_{i=0}^m(-1)^i \cdot \frac{2i+1}{2m+1}\binom{2m+1}{m-i}q^{\binom{i+1}{2}},$$
where $M$ runs through all perfect matchings on $2m$ points.
Since isolated points have no effect on the crossing number, we can easily obtain the generating function for all the matchings on $n$ points (not necessarily perfect), as follows.
Let $T_{n,j}(q)$ be the generating function for the number of matchings on $n$ points with $j$ chords with respect to the crossing number.
Then we clearly  have $p= n -2j$
unmatched points,  and thus we have:
$T_{n,j}(q)=\binom{n}{2j}T_{2j}(q)$.

\subsection{Motzkin numbers and Motzkin triangle}
\label{subsec:motzkin numbers}
Matchings $m$ satisfying ${\rm cr}(m)=0$, i.e. without crossings, deserve special attention as they are counted by the {\it Motzkin numbers}, as originally introduced by Motzkin \cite{Motzkin}. The Motzkin numbers have other various combinatorial interpretations, see \cite{DS} and sequence A001006 in OEIS \cite{OEIS}.
Their defining recursion is:
$$M_n= M_{n-1}+ \sum\limits_{i=0}^{n-2}M_i M_{n-2-i} = \frac{2n+1}{n+2}M_{n-1} + \frac{3n-3}{n+2}M_{n-2},$$
and the first numbers are
$$M_0=1,\ M_1=1,\ 2,\ 4,\ 9,\ 21,\ 51,\ 127,\ 323.$$

\begin{example}\label{Example $M_4$}
There are $M_4=9$ ways to draw non-crossing chords between $4$ points on the circle, as presented in Figure \ref{fig:M4}.

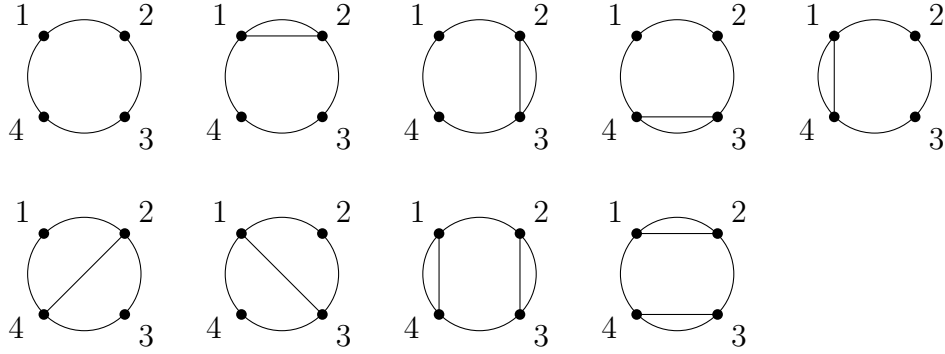
\begin{figure}[H]
\begin{center}

    \begin{tikzpicture}[
Circ/.style={draw,shape=circle,minimum size=15mm, node contents={}}
                        ]
\node (C1) [Circ];
\fill[black]    (C1.north west) circle (2pt)
                (C1.south west) circle (2pt)
                (C1.south east) circle (2pt)
                (C1.north east) circle (2pt);
\node[  shift={(100+35:0.4)},minimum size=5mm,
         ] at (C1.north west){\rm 1};
\node[  shift={(370+35:0.4)},minimum size=5mm,
         ] at (C1.north east){\rm 2};
\node[  shift={(170+35:0.4)},minimum size=5mm,
         ] at (C1.south west){\rm 4};
\node[  shift={(280+35:0.4)},minimum size=5mm,
         ] at (C1.south east){\rm 3};

\node (C2) [Circ,right=11mm of C1];
\draw[black]     (C2.north west) -- (C2.north east);
\fill[black]    (C2.north west) circle (2pt)
                (C2.south west) circle (2pt)
                (C2.south east) circle (2pt)
                (C2.north east) circle (2pt);
\node[  shift={(100+35:0.4)},minimum size=5mm,
         ] at (C2.north west){\rm 1};
\node[  shift={(370+35:0.4)},minimum size=5mm,
         ] at (C2.north east){\rm 2};
\node[  shift={(170+35:0.4)},minimum size=5mm,
         ] at (C2.south west){\rm 4};
\node[  shift={(280+35:0.4)},minimum size=5mm,
         ] at (C2.south east){\rm 3};
\node (C3) [Circ,right=11mm of C2];
\draw[black]     (C3.south east) -- (C3.north east);
\fill[black]    (C3.north west) circle (2pt)
                (C3.south west) circle (2pt)
                (C3.south east) circle (2pt)
                (C3.north east) circle (2pt);
\node[  shift={(100+35:0.4)},minimum size=5mm,
         ] at (C3.north west){\rm 1};
\node[  shift={(370+35:0.4)},minimum size=5mm,
         ] at (C3.north east){\rm 2};
\node[  shift={(170+35:0.4)},minimum size=5mm,
         ] at (C3.south west){\rm 4};
\node[  shift={(280+35:0.4)},minimum size=5mm,
         ] at (C3.south east){\rm 3};
\node (C4) [Circ,right=11mm of C3];
\draw[black]     (C4.south west) -- (C4.south east);
\fill[black]    (C4.north west) circle (2pt)
                (C4.south west) circle (2pt)
                (C4.south east) circle (2pt)
                (C4.north east) circle (2pt);
\node[  shift={(100+35:0.4)},minimum size=5mm,
         ] at (C4.north west){\rm 1};
\node[  shift={(370+35:0.4)},minimum size=5mm,
         ] at (C4.north east){\rm 2};
\node[  shift={(170+35:0.4)},minimum size=5mm,
         ] at (C4.south west){\rm 4};
\node[  shift={(280+35:0.4)},minimum size=5mm,
         ] at (C4.south east){\rm 3};
\node (C5) [Circ,right=11mm of C4];
\draw[black]     (C5.north west) -- (C5.south west);
\fill[black]    (C5.north west) circle (2pt)
                (C5.south west) circle (2pt)
                (C5.south east) circle (2pt)
                (C5.north east) circle (2pt);
\node[  shift={(100+35:0.4)},minimum size=5mm,
         ] at (C5.north west){\rm 1};
\node[  shift={(370+35:0.4)},minimum size=5mm,
         ] at (C5.north east){\rm 2};
\node[  shift={(170+35:0.4)},minimum size=5mm,
         ] at (C5.south west){\rm 4};
\node[  shift={(280+35:0.4)},minimum size=5mm,
         ] at (C5.south east){\rm 3};
\node (C6) [Circ,below=11mm of C1];
\draw[black]     (C6.south west) -- (C6.north east);
\fill[black]    (C6.north west) circle (2pt)
                (C6.south west) circle (2pt)
                (C6.south east) circle (2pt)
                (C6.north east) circle (2pt);
\node[  shift={(100+35:0.4)},minimum size=5mm,
         ] at (C6.north west){\rm 1};
\node[  shift={(370+35:0.4)},minimum size=5mm,
         ] at (C6.north east){\rm 2};
\node[  shift={(170+35:0.4)},minimum size=5mm,
         ] at (C6.south west){\rm 4};
\node[  shift={(280+35:0.4)},minimum size=5mm,
         ] at (C6.south east){\rm 3};
\node (C7) [Circ,right=11mm of C6];
\draw[black]     (C7.north west) -- (C7.south east);
\fill[black]    (C7.north west) circle (2pt)
                (C7.south west) circle (2pt)
                (C7.south east) circle (2pt)
                (C7.north east) circle (2pt);
\node[  shift={(100+35:0.4)},minimum size=5mm,
         ] at (C7.north west){\rm 1};
\node[  shift={(370+35:0.4)},minimum size=5mm,
         ] at (C7.north east){\rm 2};
\node[  shift={(170+35:0.4)},minimum size=5mm,
         ] at (C7.south west){\rm 4};
\node[  shift={(280+35:0.4)},minimum size=5mm,
         ] at (C7.south east){\rm 3};
\node (C8) [Circ,right=11mm of C7];
\draw[black]     (C8.north west) -- (C8.south west);
\draw[black]      (C8.north east) -- (C8.south east);
\fill[black]    (C8.north west) circle (2pt)
                (C8.south west) circle (2pt)
                (C8.south east) circle (2pt)
                (C8.north east) circle (2pt);
\node[  shift={(100+35:0.4)},minimum size=5mm,
         ] at (C8.north west){\rm 1};
\node[  shift={(370+35:0.4)},minimum size=5mm,
         ] at (C8.north east){\rm 2};
\node[  shift={(170+35:0.4)},minimum size=5mm,
         ] at (C8.south west){\rm 4};
\node[  shift={(280+35:0.4)},minimum size=5mm,
         ] at (C8.south east){\rm 3};
\node (C9) [Circ,right=11mm of C8];
\draw[black]     (C9.south east) -- (C9.south west);
\draw[black]      (C9.north east) -- (C9.north west);
\fill[black]    (C9.north west) circle (2pt)
                (C9.south west) circle (2pt)
                (C9.south east) circle (2pt)
                (C9.north east) circle (2pt);
\node[  shift={(100+35:0.4)},minimum size=5mm,
         ] at (C9.north west){\rm 1};
\node[  shift={(370+35:0.4)},minimum size=5mm,
         ] at (C9.north east){\rm 2};
\node[  shift={(170+35:0.4)},minimum size=5mm,
         ] at (C9.south west){\rm 4};
\node[  shift={(280+35:0.4)},minimum size=5mm,
         ] at (C9.south east){\rm 3};
\end{tikzpicture}
\end{center}
\caption{$9$ ways to draw non-crossing chords between $4$ points on the circle.}\label{fig:M4}
\end{figure}

\end{example}

\medskip

There is a refinement of Motzkin numbers, called {\it Motzkin triangle} (see sequence A055151 in the OEIS \cite{OEIS}), and denoted $M_{n,k}$, whose first 9 lines are presented in Table \ref{tab:A055151}.

\begin{table}[H]
$${\small
\begin{NiceArray}{|w{c}{0.4cm}|c;c;c;c;c|}
\hline
\Gape[0.35cm][0.35cm]{\diagbox{\,\,n}{k\,}} & 0 & 1 & 2 & 3 & 4\\
\hline
0  &  1  & &&& \\
1  &  1  & &&& \\
2  &  1  & 1 &&& \\
3  &  1 &  3 &&&\\
4  &  1 & 6 & 2 &&\\
5  &  1 & 10 & 10 &&\\
6  &  1 & 15 & 30 & 5 & \\
7  &  1 & 21 & 70 & 35& \\
8  &  1 & 28 & 140 & 140 & 14\\
\hline
\end{NiceArray}}$$

\caption{Table form of sequence A055151}
\label{tab:A055151}
\end{table}

\medskip

The element $M_{n,k}$ in the Motzkin triangle  counts the number of ways to draw $k$ non-crossing chords between $n$ points on the circle. For example, the fifth line  of the triangle, ``1\ 6\ 2'', counts one diagram of $4$ points on a circle without chords, $6$ diagrams  with one chord and $2$ diagrams with 2 chords (see Figure  \ref{fig:M4} above).

\begin{remark}
Note that the combinatorial explicit connection between Touchard-Riordan polynomials and Motzkin numbers is given by $T_{2n}(0)=M_{2n}$. We use this observation in Corollary \ref{specializations}(3)--(4) below.
\end{remark}

\section{Schur-positivity of various set partitions with respect to ${\rm Short}$}\label{section on schur positivity for des sharing block}
Let ${\mathcal{PS}}et(n,b)$ be the set of all set partitions of the set $[n]$ with $b$ blocks. For\break $\pi \in {\mathcal{PS}}et(n,b)$, define: $${\rm Short}(\pi)=\{i \in [n-1] \mid i \mbox{ and } i+1 \mbox { are in the same block} \}.$$
Note that this set parameter is not sparse in the sense of Definition 1.7 of Marmor \cite{Marmor}.

\medskip

Our first main result deals with Schur-positivity with respect to the set parameter ${\rm Short}$ we have just defined:

\begin{thm}
\begin{equation}\label{q touchard}
\sum\limits_{\pi \in {\mathcal{PS}}et(n,b)} q^{{\rm cr} (\pi)} \mathcal{F}_{n,{\rm Short}(\pi)} =\sum\limits_{k=0}^{n-1} T_{n-1-k,n-k-b}(q) s_{(n-k,1^k)}
\end{equation}
\end{thm}

\begin{proof}
Let $\mathcal{D}_{n,j}$ be the set of chord diagrams having $n$ points and $j$ chords, and let $\mathcal{PS}et(n,b,k)$ be the set of set partitions of $[n]$ having $b$ blocks and $k$ descents.

One can consider both sides of Equation (\ref{q touchard}) as polynomials in the variable $q$, whose coefficients are in the ring of quasi-symmetric functions. To prove the equality, we have to show that the corresponding coefficients are equal. We do that with the aid of Theorem \ref{gessel}, by invoking a bijection
$${\rm SYT}(n-k,1^k)\times \mathcal{D}_{n-1-k,n-k -b} \to \mathcal{PS}et(n,b,k),$$ which preserves the crossing number.

We start with the case $k=0$, i.e., we present a bijection from the set of pairs $$P=\left\{\left(\raisebox{-.4ex}{\begin{smallytableau} {1}&{2}&\cdots& n  \end{smallytableau}}\ , \ C\right) \mid C \in \mathcal{D}_{n-1,n-b}\right\}$$ to $\mathcal{PS}et(n,b,0)$
which, as defined above, is the set of set partitions $\pi$ of $[n]$, having $b$ blocks and ${\rm Short}(\pi)=\emptyset$. Let $C\in \mathcal{D}_{n-1,n-b}$, i.e., a chord diagram on $n-1$ points with $n-b$  chords. The bijection has two steps: First, for each chord $\{i,j\}$ of $C$ with $i<j$, we create the block $\{i,j+1\}$. Next, we unify every two intersecting blocks. The points unmatched by chords contribute singletons to the resulting set partition.

Since the crossing number depends solely on the arcs that connect consecutive elements in the blocks of the partition, it is easy to see that the resulting set partition, consisting of $b$ blocks, has the same number of crossings as in the chord diagram.

In the opposite direction, when we are given a set partition $\pi$ with ${\rm Short}(\pi)=\emptyset$, it is straightforward to recover the chord diagram that is mapped to it.

For example, for the chord diagram in the left part of Figure \ref{example of bijection_cross}, which is an element of $\mathcal{D}_{7,3}$, we first construct the blocks $\{1,4\},\{2,6\}, \{4,8\}$, and then we send it to the set partition $$148\mid 26\mid 3\mid 5\mid 7 \in \mathcal{PS}et (8,5,0).$$
Note that there are two crossings in both the chord diagram and the corresponding set partition, see Figure \ref{example of bijection_cross}.

\begin{figure}[H]

\hspace{-250pt}\vspace{-80pt}\begin{tikzpicture}
\node[draw=none,rotate=45, minimum size=3cm,regular polygon,regular polygon sides=7] (a) {};
  \fill (a.corner 1) circle[radius=2pt] node[shift={(72+35:0.4)}] {1};
  \fill (a.corner 2) circle[radius=2pt] node[shift={(144+35:0.4)}] {7};
  \fill (a.corner 3) circle[radius=2pt] node[shift={(216+35:0.4)}] {6};
  \fill (a.corner 4) circle[radius=2pt] node[shift={(288+35:0.4)}] {5};
  \fill (a.corner 5) circle[radius=2pt] node[shift={(35:0.4)}] {4};
  \fill (a.corner 6) circle[radius=2pt] node[shift={(35:0.4)}] {3};
  \fill (a.corner 7) circle[radius=2pt] node[shift={(35:0.4)}] {2};

\draw[black]     (a.corner 1) -- (a.corner 6);
\draw[black]     (a.corner 7) -- (a.corner 4);
\draw[black]     (a.corner 5) -- (a.corner 2);
\draw (0,0) circle (1.5cm);
\end{tikzpicture}

\hspace{150pt}\begin{tikzpicture}[scale=1]

  \foreach \x in {1,...,8} {
    \fill (\x,0) circle (2pt);
    \node[below] at (\x,0) {\x};
  }

\fill[gray] (2.72,0.69) circle (2pt);
\fill[gray] (5.37,0.66) circle (2pt);

  \def\arc#1#2#3{
    \draw (#1,0) arc[start angle=180, end angle=0, x radius={(#2-#1)/2}, y radius=#3];
  }

  \arc{1}{4}{0.7}
  \arc{4}{8}{0.7}
  \arc{2}{6}{0.9}
\end{tikzpicture}

\vspace{30pt}
\caption{Left: A chord diagram. Right: The linear graph of the resulting set partition.}
\label{example of bijection_cross}
\end{figure}
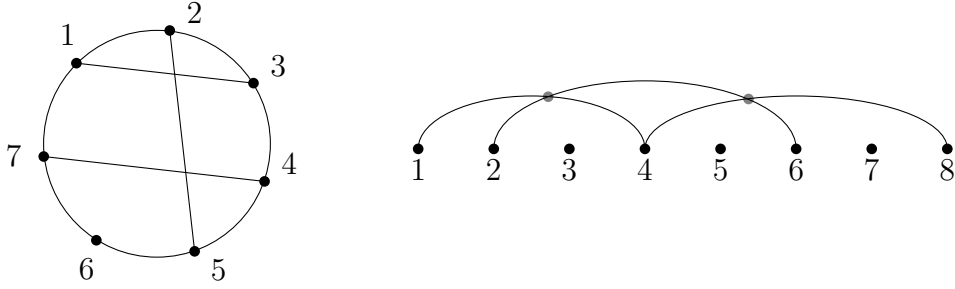

We turn now to the case $k \geq 1$. Clearly, every standard Young tableau of shape $(n-k,1^k)$ has exactly $k$ descents. Given a pair $(T,C)\in {\rm SYT}(n-k,1^k) \times \mathcal{D}_{n-1-k,n-k -b}$, we first define a function $\varphi:[n]\rightarrow [n-k]$ by:
$$\varphi(j)=j-|\{i\in {\rm Des}(T) \mid i<j\}|.$$
Next, we apply the procedure described above for the case $k=0$ for the chord diagram $C$ having $n-1-k$ points on a circle with $n-k-b$ chords, resulting in a partition of the set $[n-k]$, having $b$ blocks. Finally, we obtain a set partition of $[n]$ by replacing each element of $[n-k]$ by the elements in its preimage by $\varphi$, see Example \ref{example bijection k descents} below.

\medskip

Note that in the last step, while inflating the set $[n-k]$ to the set $[n]$ using $\varphi$, the crossing number of the set partition does not change, as we count only crossings of arcs, while the inflation contributes only straight segments.

\medskip

In the opposite direction, note that when we are given a set partition $\pi$, one can easily recover the sets of the form ${\rm Short}(\pi)$, which induce the descent set of the tableau, and the associated function $\varphi$. Then, it is straightforward to recover also the chord diagram which is mapped to it.
\end{proof}

\begin{example}\label{example bijection k descents}
Given a standard Young tableau $T$, with ${\rm Des}(T)=\{1,2,5,7,9\}\subseteq [13]$ and a chord diagram $C$ as in Figure \ref{example of bijection_cross} above, we first define
$\varphi: [13] \to [8]$, using the tableau $T$, as follows: $$\varphi(1)=\varphi(2)=\varphi(3)=1, \ \varphi(4)=2,\ \varphi(5)=\varphi(6)=3,\ \varphi(7)=\varphi(8)=4,$$
$$\varphi(9)=\varphi(10)=5,\ \varphi(11)=6, \ \varphi(12)=7 \mbox{ and } \varphi(13)=8.$$
Next, we use the chord diagram $C$ to obtain the set partition, as computed in the case $k=0$:
$$148\mid 26 \mid 3 \mid 5 \mid 7.$$
Finally, we replace each element of the set $[8]$ by the elements in its preimage by $\varphi$, to obtain the set partition:
$$\{\{\underbrace{1,2,3}_{\varphi^{-1}(1)},\underbrace{7,8}_{\varphi^{-1}(4)},\underbrace{13}_{\varphi^{-1}(8)}\},\{\underbrace{4}_{\varphi^{-1}(2)},\underbrace{11}_{\varphi^{-1}(6)}\},\{\underbrace{5,6}_{\varphi^{-1}(3)}\},\{\underbrace{9,10}_{\varphi^{-1}(5)}\},
\{\underbrace{12}_{\varphi^{-1}(7)}\}\}\in \mathcal{PS}et(13,5,5).$$
\end{example}

\bigskip

Plugging $q=1$ in Equation (\ref{q touchard}) induces an explicit Schur-positivity result for the set of set partitions, while substituting $q=0$ reduces us to a Schur-positivity result for non-crossing partitions, where elements from the Motzkin triangle replace the  coefficients of the Touchard-Riordan polynomials, so we have the following:

\begin{theorem}\label{specializations}\ \\

\begin{enumerate}
\item $\mathcal{Q}_n({\mathcal{PS}}et(n,b))=\sum\limits_{k=0}^{n-1} T_{n-1-k,n-b-k}(1) s_{(n-k,1^k)}$
\item $\mathcal{Q}_n({\mathcal{PS}}et(n))=\sum\limits_{k=0}^{n-1} \sum\limits_{b=1}^{n-k} T_{n-1-k,n-b-k}(1) s_{(n-k,1^k)}$
\item $\mathcal{Q}_n({\mathcal NC}(n,b))=\sum\limits_{k=0}^{n-1} M_{n-1-k,n-b-k} s_{(n-k,1^k)}$
\item $\mathcal{Q}_n({\mathcal NC}(n))=\sum\limits_{k=0}^{n-1} M_{n-k-1} s_{(n-k,1^k)}$.
\end{enumerate}

\end{theorem}

\section{The Schur-positivity of  $\mathcal{PS}et(n)$ with respect to the parameter ${\rm Desing}$}\label{subsection schur positivity with respet to desing}

In this section, we deal with a second result of Schur-positivity of $\mathcal{PS}et(n)$ with respect to a different parameter, denoted Desing, which is defined for a given set partition as follows.

\begin{definition}\label{first definitions for Desing}
For a set partition $\pi\in \mathcal{PS}et(n)$, we denote by ${\rm Sing }(\pi)$ the set of singletons of  $\pi$.

Let $${\rm Desing}(\pi)=\{i \mid i \in Sing(\pi),i+1 \notin Sing(\pi) \}=\{d_1< \cdots <d_s\},$$ and define:
${\rm desing}(\pi)=|{\rm Desing}(\pi)|=s$.

We will also use the set: ${\rm SucDes}(\pi)=\{i+1 \mid i \in {\rm Desing}(\pi)\}$.
\end{definition}

Note that the ${\rm Desing}$ parameter is {\it sparse} in the sense of Definition 1.7 of Marmor \cite{Marmor}.

\medskip

Moreover, we define:
\begin{definition}\label{def of h_pi}
\begin{equation}
H(\pi)=[n] \, \backslash \, ([1,d_1]  \cup {\rm Sing}(\pi) \cup {\rm SucDes}(\pi)).
\end{equation}

For each $i \in {\rm Desing}(\pi)$, let $\ell\geq 0$ be maximal such that $i-\ell,i-\ell+1,\dots,i-1$ are singletons in $\pi$ preceding $i$. Such a sequence will be called {\em truncated run} and let $t_i=\ell$, so that $t_i$ is the number of consecutive singletons preceding $i$.

For a given $h \in H(\pi)$, note that $d_i < h < d_{i+1}$ for some $0\leq i \leq s-1$ (where $d_0=0$).  Then, define recursively
$$m_h=m_h({\pi})=\sum\limits_{j=1}^i t_{d_j}-\sum\limits_{\{g\in H(\pi)\, \mid \, g<h\}}\chi(m_g>0),$$ where $\chi(P)$ is the characteristic function of the property $P$.
\end{definition}

\begin{example}\label{first example}
Given the set partition
$$\pi=\{\{1\}, \{2\},\{3, 4,5,9,10,11,14\},\{ 6\},\{ 7\},\{8\},\{ 12\}, \{ 13\}\}, $$
we have that $${\rm Sing}(\pi)=\{1,2,6,7,8,12,13\},\  {\rm Desing}(\pi)=\{2,8,13\},\  {\rm SucDes}(\pi)=\{3,9,14\},$$ the truncated run of $2$ is $1$, the truncated run of $8$ is $6,7$ and the truncated run of $13$ is $12$, so that $t_2=1,\ t_8=2$, and $t_{13}=1$.
Moreover, we have $H(\pi)=\{4,5,10,11\}$,
Recall the definition of $m_h$ for $h\in H(\pi)$:
$$m_4= 1-0=1, \ m_5 =1-1=0, \ m_{10}=1+2-1=2, \ m_{11}=1+2-2=1.$$

\end{example}

\begin{definition}\label{def associated Bell numbers}
Let $\left\{{\rm Bell}_{\geq 2}(n)\right\}$ be the sequence of the well-known {\em associated Bell numbers}, counting the set partitions of $[n]$ without singletons. This is sequence A000296 in OEIS \cite{OEIS}:
$$1,0,1,1,4,11,41,162,715,3425,\dots$$
\end{definition}

In this section, we prove the following result:
\begin{thm}\label{main theorem}
Define for $n\geq0$:
$$e_{n-k,k} = \left\{ \begin{array}{ll}
\sum\limits_{i=2}^{n-1} {\rm Bell}_{\geq 2}(i) & k=1\\
\sum\limits_{i=k}^{n-k} {\rm Bell}_{\geq 2}(i) & (k=0) \mbox{ or } (k>1). \\
\end{array}\right.$$
Then we have:
\begin{equation}
\sum\limits_{\pi \in {\mathcal {PS}}et(n)}{\mathbf q}^{{\rm Desing}(\pi)}=\sum\limits_{k=0}^{\left\lfloor \frac{n}{2}\right\rfloor} e_{n-k,k}
\left(\sum\limits_{T \in {\rm SYT}(n-k,k)}{\mathbf q}^{{\rm Des}(T)}\right),\end{equation}
where for each $A\subseteq [n-1]$, we define ${\mathbf q}^A=q_1^{\epsilon_1}\cdots q_{n-1}^{\epsilon_{n-1}}$ and $\epsilon_i=\left\{\begin{array}{ll}
1  & i \in A\\
0 & i \notin A \end{array}\right.,$ and ${\rm Des}(T)$ was defined in Definition \ref{Des of T} above.

We conclude, using Gessel's theorem (Theorem \ref{gessel} above), that
\begin{equation}
    \mathcal{Q}_n(\mathcal{PS}et(n))=\sum\limits_{k=0}^{\left\lfloor\frac{n}{2}\right\rfloor} e_{n-k,k} s_{(n-k,k)}.
    \end{equation}
\end{thm}

A combinatorial connection between the length of the second row of a shape $(n-k,k)$ and the corresponding set partitions will be explained in Corollary \ref{cor: combinatorial meaning k} below.

\medskip

We immediately have the following simple enumerative result:
\begin{cor}\label{coro:bell sum}
Let $B(n)=|\mathcal{PS}et(n)|$ be the celebrated ordinary {\em Bell number}. Then, we have $$B(n)=\sum\limits_{k=0}^{\left\lfloor\frac{n}{2}\right\rfloor} e_{n-k,k} c_{n-k,k},$$
where $c_{n-k,k}$
are the {\em ballot numbers} (sequence A009766 in OEIS \cite{OEIS}), counting the number of standard Young tableaux of shape $(n-k,k)$ (see also Section \ref{visual 2} below).
\end{cor}

\begin{remark}
A decomposition of the Bell number $B(n)$ as a sum of terms indexed by two-row Young diagrams can be found implicitly in the representation-theoretic framework of Stanley \cite[Section 7.18]{EC2}. Our contribution here is an explicit combinatorial realization of this decomposition via a bijection with standard Young tableaux of shape $(n-k,k)$, together with a closed formula for the corresponding multiplicities.
\end{remark}

\medskip

The proof of Theorem \ref{main theorem} consists of the following parts. First, we compute the number of set partitions $\pi$ with ${\rm desing}(\pi)=0$ (Claim \ref{claim: partitions without descents}). Next, we present the mapping from set partitions to Standard Young Tableaux (Section \ref{section: mapping desing}), and then we show that the size of the pre-image of a given tableau depends only on its shape and compute it explicitly (Section \ref{section: preimage desing}).

\medskip

We start with the following observation:
\begin{claim}\label{claim: partitions without descents}
The number of set partitions $\pi  $ of $[n]$ having ${\rm desing}(\pi)=0$ is: $$e_{n,0}=\sum\limits_{i=0}^n {\rm Bell}_{\geq 2}(i).$$
\end{claim}
\begin{proof}
For constructing a set partition $\pi$  having ${\rm desing}(\pi)=0$, we have to locate the singletons after the blocks which are not singletons, and hence the result follows (where $n-i$ is the number of possible singletons).
\end{proof}

Note that $\left\{e_{n,0}\right\}_{n \in \mathbb{N}}$ is sequence
A160181 in OEIS \cite{OEIS}:
$$1, 1, 2, 3, 7, 18, 59, 221, 936, 4361,\dots$$

\subsection{The mapping from set partitions to Standard Young Tableaux with two rows}\label{section: mapping desing}

We define a mapping
$$\varphi:\mathcal{PS}et(n) \longrightarrow \bigcup\limits_{\scriptsize\begin{array}{c}\lambda = (n-k,k) \\ \ 0\leq k \leq \lfloor\frac{n}{2}\rfloor \end{array}} {\rm SYT}(\lambda),$$
by the following rules, for a given set partition $\pi$:

\begin{enumerate}
\item[{\bf Rule (0):}] If ${\rm desing}(\pi)=0$, then $\varphi(\pi)=\raisebox{-.4ex}{\begin{smallytableau} 1&2&\cdots&n \end{smallytableau}}$ .

\medskip

\item[{\bf Rule (1):}] Otherwise, locate every element of ${\rm Sing}(\pi)$ and every element smaller than\break $d_1=\min({\rm Desing}(\pi))$ in the first row of $\varphi(\pi)$. Locate every element of the set ${\rm SucDes}(\pi)$ in the second row of $\varphi(\pi)$. Order the elements in both rows increasingly.
Note that this rule locates all the elements in $[n]\backslash H(\pi)$ in $\varphi(\pi)$.
\medskip

\item[{\bf Rule (2):}]
Recall the definition of $m_h$ from Definition \ref{def of h_pi}. For each  $h \in H(\pi)$, we add it   to $\varphi(\pi)$ as follows:
If $m_h>0$, we locate $h$ in the second row of $\varphi(\pi)$ and otherwise, we locate it in the first row of $\varphi(\pi)$, where in both cases the location of the element $h$ in the row preserves the ascending order of that row.
\end{enumerate}

\medskip

Note that by Rule (1), the mapping we described satisfies $Desing(\pi)=Des(\varphi(\pi))$.

\begin{remark}\label{rem-pairing}
Here is a more visual way to describe the mapping presented above. We set a stack that will hold singletons which are not elements of ${\rm Desing}(\pi)$. For each $h\in H(\pi)$, let $d\in {\rm Desing}(\pi)$ be maximal such that $d<h$, and insert the elements of all truncated runs preceding $d$ increasingly into the stack. At the end of scanning $\pi$, insert into the stack the rest of the truncated runs.

Now, the decision made in Rule (2), whether to locate $h$ in the first row or in the second row is simple: locate $h$ in the second row if and only if the stack is not empty at this stage. In case the element $h$ was placed in the second row, pop out the top element of the stack.
Note that this defines an injective mapping between the elements of $H(\pi)$ located in the second row of $\varphi(\pi)$ and the singletons from the first row popped up from the stack.
\end{remark}

We illustrate the map $\varphi$ by some examples:

\begin{example}\label{cont first example}
This example continues Example \ref{first example} above. We have the set partition
$$\pi=\{\{1\}, \{2\},\{3, 4,5,9,10,11,14\},\{ 6\},\{ 7\},\{8\},\{ 12\}, \{ 13\}\}. $$
Recall that the truncated run of $2$ is $1$, the truncated run of $8$ is $6,7$ and the truncated run of $13$ is $12$.
Moreover, we have $H(\pi)=\{4,5,10,11\}$, and
$$m_4>0, \ m_5=0, \ m_{10}>0 \mbox{ and } m_{11}>0.$$

In the stack model (see Figure \ref{fig:stack example}):
\begin{itemize}
\item For $h=4$ (see step $1$), we insert the truncated run smaller than $4$ (which is only $1$), and then we pop it out. So we have the correspondence $1 \leftrightarrow 4$, and we locate $h=4$ in the second row.

\item For $h=5$ (see step $2$), we have nothing to insert to the stack, so the stack is empty and there is no pop-up. Therefore, $h=5$ is located in the first row.

\item For $h=10$ (see step $3$), we insert the truncated run smaller than $10$ (which is $6,7$), and then we pop $7$ out. So we have the correspondence $7 \leftrightarrow 10$, and we locate $h=10$ in the second row.

\item For $h=11$ (see step $4$), we have nothing to insert into the stack. Since the stack is not empty, we pop $6$ out. So we have the correspondence $6 \leftrightarrow 11$, and we locate $h=11$ in the second row.

\item While finishing the scanning of $\pi$ (see step $5$), we insert the remaining truncated run (which is only $12$) to the stack.
\end{itemize}

\medskip

Hence, we get the following SYT:
$$\varphi(\pi)={\begin{ytableau} 1&2&5&6&7&8&12&13 \\ 3&4&9&10&11&14 \end{ytableau}} \ .$$

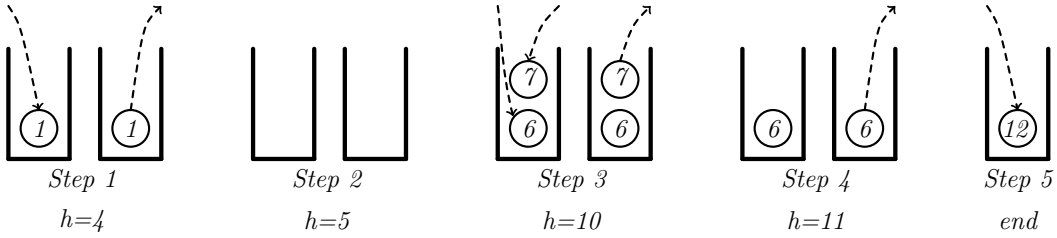
\begin{figure}[H]
    \centering
\resizebox{0.9\textwidth}{!}{\begin{tikzpicture}[cap=round,line width=2pt]

\foreach \i in
    {1, 2, 3, 4}
  {
    \draw[-] (4*\i-2,0) -- (1+4*\i-2,0);
    \draw[-] (4*\i-2,0) -- (0+4*\i-2,1.8);
    \draw[-] (4*\i-1,0) -- (4*\i-1,1.8);
    \node[below] at (0.7+4*\i-3,0){Step \i};
  }

\foreach \i in
    {1, 2, 3, 4,5}
  {
    \draw[-] (4*\i-3.5,0) -- (1+4*\i-3.5,0);
    \draw[-] (4*\i-3.5,0) -- (0+4*\i-3.5,1.8);
    \draw[-] (4*\i-2.5,0) -- (4*\i-2.5,1.8);
  }

\draw[line width=1pt]  (1,0.5) circle (0.3cm);

\node at (1,0.5){1};

\draw[line width=1pt]  (2.5,0.5) circle (0.3cm);

\node at (2.5,0.5){1};

\node at (1.7,-1){h=4};

\node at (5.7,-1){h=5};

\draw[line width=1pt]  (9,0.5) circle (0.3cm);

\node at (9,0.5){6};

\draw[line width=1pt]  (9,1.3) circle (0.3cm);

\node at (9,1.3){7};

\draw[line width=1pt]  (10.5,0.5) circle (0.3cm);

\node at (10.5,0.5){6};

\draw[line width=1pt]  (10.5,1.3) circle (0.3cm);

\node at (10.5,1.3){7};

\node at (9.7,-1){h=10};

\draw[line width=1pt]  (13,0.5) circle (0.3cm);

\node at (13,0.5){6};

\draw[line width=1pt]  (14.5,0.5) circle (0.3cm);

\node at (14.5,0.5){6};
\node at (13.7,-1){h=11};

\node[below] at (17,0){Step 5};
\node at (17,-1){end};

\draw[line width=1pt]  (17,0.5) circle (0.3cm);

\node at (16.95,0.5){12};

\draw[line width=1pt,dashed,->] (0.5,2.5) .. controls (0.7,2.2)  .. (1,0.8);

\draw[line width=1pt,dashed,->] (2.5,0.8) .. controls (2.7,2.2)  .. (3,2.5);

\draw[line width=1pt,dashed,->] (8.5,2.5) .. controls (8.55,2.2) and (8.55,1.5) .. (8.75,0.7);

\draw[line width=1pt,dashed,->] (9.5,2.5) .. controls (9.2,2.2)  .. (9,1.6);

\draw[line width=1pt,dashed,->] (10.5,1.6) .. controls (10.7,2.2)  .. (11,2.5);

\draw[line width=1pt,dashed,->] (14.5,0.8) .. controls (14.7,2.2)  .. (15,2.5);

\draw[line width=1pt,dashed,->] (16.5,2.5) .. controls (16.7,2.2)  .. (17,0.8);

\end{tikzpicture}}
\caption{Using the stack in Example \ref{cont first example}}\label{fig:stack example}
\end{figure}

\end{example}

\begin{example}\label{example partition to table}

\begin{enumerate}
\item  Let $\pi=1\mid 23\mid 4 \mid 5\mid 678$. We have: ${\rm Sing}(\pi)=\{1,4,5\}$, ${\rm Desing}(\pi)=\{1,5\}$ and ${\rm SucDes}(\pi)=\{2,6\}$. Hence, by Rule (1), the elements $1,4,5$ take their guaranteed places in the first row of $\varphi(\pi)$ and the elements $2,6$ are placed in the second row of $\varphi(\pi)$.

In the passage to Rule (2), we have $H(\pi)=\{3,7,8\}$, $t_1=0$ and $t_5=1$.

\begin{itemize}
\item For $h=3\in H(\pi)$, we have: $m_3=t_1-0=0$ which is not positive, and so we locate $3$ in the first row.
\item For $h=7\in H(\pi)$, we have: $m_7=t_1+t_5-0=1$ which is positive, and so we locate $7$ in the second row.
\item For $h=8\in H(\pi)$, we have: $m_8=t_1+t_5-1=0$ which is not positive, and so we locate $8$ in the first row.
\end{itemize}

\medskip

Hence, we get the following SYT:
$$\varphi(\pi)={\begin{ytableau} 1&3&4&5&8 \\ 2&6&7 \end{ytableau}} \ .$$
To further elaborate on the underlying reasoning of our algorithm, in the spirit of Remark \ref{rem-pairing},  note that $i=5 \in {\rm Desing}(\pi)$ is preceded by the singleton\break $s=4\in {\rm Sing} (\pi)$, and hence we added the element $h=7$ in the second row.\\ Note also that having the sequence $3,4,5$ in the first row might have implied locating $8$  in the second row. However, this apparent implication is not correct, since only $4$ is a singleton in the set partition (and not $3$), and hence we still have the sequence $6,7$ in the second row, but $8$ returns to the first row.

\medskip

\item  Let $\pi=1\mid 2\mid 356 \mid 4 \mid 7$. We have: ${\rm Sing}(\pi)=\{1,2,4,7\}$, ${\rm Desing}(\pi)=\{2,4\}$, and ${\rm SucDes}(\pi)=\{3,5\}$.
Hence, by Rule (1), the elements $1,2,4,7$ take their guaranteed places in the first row of $\varphi(\pi)$ and the elements $3,5$ are placed in the second row of $\varphi(\pi)$.

In the passage to Rule (2), we have $H(\pi)=\{6\}$, $t_2=1$ and $t_4=0$. For the single element $h=6\in H(\pi)$, we have: $m_6=t_2+t_4-0=1$ which is positive, and so we locate $6$ in the second row, and we get:
$$\varphi(\pi)={\begin{ytableau} 1&2&4&7 \\ 3&5&6 \end{ytableau}} \ .$$
Note that $2 \in {\rm Desing}(\pi)$ is preceded by the singleton $1 \in {\rm Sing}(\pi)$, but since $4$ is a descent, we have to put it in the first row, and therefore we add the element $6$ in the second row, as a debt for the `missing' successor of the first descent.

\medskip

\item Let $\pi=1 \mid 27\mid 3\mid 4\mid 5 \mid 6$.
We have: ${\rm Sing}(\pi)=\{1,3,4,5,6\}$,\break ${\rm Desing}(\pi)=\{1,6\}$, and ${\rm SucDes}(\pi)=\{2,7\}$.
Hence, by Rule (1), the elements $1,3,4,5,6$ are in the first row of $\varphi(\pi)$ and $2,7$ are in the second row of $\varphi(\pi)$.
In the passage to Rule (2), we have $H(\pi)=\emptyset$, so Rule (2) is not applicable in this example. So we have:
$$\varphi(\pi)={\begin{ytableau} 1&3&4&5&6 \\ 2&7 \end{ytableau}} \ .$$

\medskip

\item  Let $\pi=1 \mid 237 \mid 4 \mid 5 \mid 6$.
We have: ${\rm Sing}(\pi)=\{1,4,5,6\}$, ${\rm Desing}(\pi)=\{1,6\}$, and ${\rm SucDes}(\pi)=\{2,7\}$.
Hence, by Rule (1), the elements $1,4,5,6$ are in the first row of $\varphi(\pi)$ and $2,7$ are in the second row of $\varphi(\pi)$.

In the passage to Rule (2), we have $H(\pi)=\{3\}$, $t_1=0$ and $t_6=2$. For the single element $h=3 \in H(\pi)$, we have: $m_3=t_1-0=0$ which is not positive, and so we locate $3$ in the first row.
Hence, we have:
$$\varphi(\pi)={\begin{ytableau} 1&3&4&5&6 \\ 2&7 \end{ytableau}} \ .$$
\end{enumerate}
\end{example}

\medskip

Now, we have the following crucial property of $\varphi(\pi)$:
\begin{claim}\label{indeed standard}
   $\varphi(\pi)$ is indeed a standard Young tableau.
\end{claim}

\begin{proof}
For the sake of the proof, we split the application of Rule (1) into two parts. We first place the elements of ${\rm Desing}(\pi)$ in the first row of $\varphi(\pi)$ and the corresponding elements of the set ${\rm SucDes}(\pi)$ in its second row. This gives us a pair of two parallel, equal horizontal blocks, one of which is in the first row and the other in the second, see the first tableau in Example \ref{example standard tableau} below.   Obviously, at the end of this step, the elements  of the tableau $\varphi(\pi)$ increase in rows and columns.

Next, according to Rule (1), we add the elements $1,\dots ,d_1-1$ and the rest of the elements of the set ${\rm Sing}(\pi)$ to the first row of $\varphi(\pi)$. This might result in splitting the horizontal block in the first row into sub-blocks, while pushing them to the right, see the second tableau in Example \ref{example standard tableau}.
This step also does not violate the conditions of being a standard Young tableau, since the new elements are placed in order and hence the elements of the first row become smaller (or equal) with respect to the original elements in the same place before the insertion.

In the final step, we insert the elements of $H(\pi)$ in their places in the first or the second row,  as prescribed by Rule (2), see the last two tableaux in Example \ref{example standard tableau}. The elements of $H(\pi)$ which have to be inserted to the first row will not violate the tableau being standard, where the explanation is similar to that of the previous step.

By Rule (2), each element of $H(\pi)$ which is destined to find its place in the second row has a corresponding singleton that was already inserted to the first row (as can be understood implicitly from Rule (2), see Remark \ref{rem-pairing}). Thus the singleton pushes the relevant elements one step rightward, and preserves the tableau being standard.
\end{proof}

\medskip

\begin{example}\label{example standard tableau}
Let $$\pi=\{\{1,3,5\},\{2\},\{4\},\{6,9,14\},\{7\} ,\{8\},\{10\},\{11,13\},\{12\}\}.$$
We have: ${\rm Sing}(\pi)=\{2,4,7,8,10,12\}$ and ${\rm Desing}(\pi)=\{2,4,8,10,12\}$, hence, by Rule (1), the elements $2,4,8,10,12$ take their guaranteed place in the first row of $\varphi(\pi)$ and the elements ${\rm SucDes}(\pi)=\{3,5,9,11,13\}$ are placed in the second row of $\varphi(\pi)$. So we have the following temporary diagram:
$$T=\begin{Tableau}{{2,4,8,10,12},{3,5,9,11,13}}
      \draw[very thick,blue] (0.1,-0.9)--(4.9,-0.9)--(4.9,-0.1)--(0.1,-0.1)--(0.1,-0.9);
      \draw[very thick,red] (0.1,-1.9)--(4.9,-1.9)--(4.9,-1.1)--(0.1,-1.1)--(0.1,-1.9);
  \end{Tableau} ,$$
which is obviously standard. Next, by Rule (1), we add the element $1$, as all the elements before the first descent should be in the first row, and the element $7\in {\rm Sing }(\pi)$, so $T$ has the following temporary design:
$$T'=\begin{Tableau}{{1,2,4,7,8,10,12},{3,5,9,11,13}}
      \draw[very thick,blue] (1.1,-0.9)--(2.9,-0.9)--(2.9,-0.1)--(1.1,-0.1)--(1.1,-0.9);
      \draw[dashed, very thick,blue] (3.1,-0.9)--(3.9,-0.9)--(3.9,-0.1)--(3.1,-0.1)--(3.1,-0.9);
      \draw[very thick,blue] (4.1,-0.9)--(6.9,-0.9)--(6.9,-0.1)--(4.1,-0.1)--(4.1,-0.9);
      \draw[very thick,red] (0.1,-1.9)--(1.9,-1.9)--(1.9,-1.1)--(0.1,-1.1)--(0.1,-1.9);
      \draw[very thick,red] (2.1,-1.9)--(4.9,-1.9)--(4.9,-1.1)--(2.1,-1.1)--(2.1,-1.9);
  \end{Tableau} ,$$
which is still standard.

Now, we have $H(\pi) = \{6,14\}$. By Rule (2), the element $6$ will be added to the first row, so we get:
$$T''=\begin{Tableau}{{1,2,4,6,7,8,10,12},{3,5,9,11,13}}
      \draw[very thick,blue] (1.1,-0.9)--(2.9,-0.9)--(2.9,-0.1)--(1.1,-0.1)--(1.1,-0.9);
      \draw[dashed, very thick,blue] (4.1,-0.9)--(4.9,-0.9)--(4.9,-0.1)--(4.1,-0.1)--(4.1,-0.9);
      \draw[very thick,blue] (5.1,-0.9)--(7.9,-0.9)--(7.9,-0.1)--(5.1,-0.1)--(5.1,-0.9);
      \draw[very thick,red] (0.1,-1.9)--(1.9,-1.9)--(1.9,-1.1)--(0.1,-1.1)--(0.1,-1.9);
      \draw[very thick,red] (2.1,-1.9)--(4.9,-1.9)--(4.9,-1.1)--(2.1,-1.1)--(2.1,-1.9);
  \end{Tableau} ,$$
which is still standard.

Again by Rule (2), the element $14$ will be added to the second row (dashed) due to the existence of the singleton $7$ (dashed):
$$\varphi(\pi)=T'''=\begin{Tableau}{{1,2,4,6,7,8,10,12},{3,5,9,11,13,14}}
      \draw[very thick,blue] (1.1,-0.9)--(2.9,-0.9)--(2.9,-0.1)--(1.1,-0.1)--(1.1,-0.9);
      \draw[dashed, very thick,blue] (4.1,-0.9)--(4.9,-0.9)--(4.9,-0.1)--(4.1,-0.1)--(4.1,-0.9);
      \draw[very thick,blue] (5.1,-0.9)--(7.9,-0.9)--(7.9,-0.1)--(5.1,-0.1)--(5.1,-0.9);
      \draw[very thick,red] (0.1,-1.9)--(1.9,-1.9)--(1.9,-1.1)--(0.1,-1.1)--(0.1,-1.9);
      \draw[very thick,red] (2.1,-1.9)--(4.9,-1.9)--(4.9,-1.1)--(2.1,-1.1)--(2.1,-1.9);
      \draw[dashed,very thick,red] (5.1,-1.9)--(5.9,-1.9)--(5.9,-1.1)--(5.1,-1.1)--(5.1,-1.9);
  \end{Tableau} ,$$
which is standard.
\end{example}

\medskip

The following is a direct consequence of Rules (1) and (2) above:

\begin{cor}\label{cor: combinatorial meaning k}
Let $\pi$ be a set partition of $[n]$. Then the shape of $\varphi(\pi)$ is $(n-k,k)$, where
$$k={\rm  desing(\pi)}+\sum\limits_{h\in H(\pi)}\chi(m_h>0).$$
\end{cor}

\medskip

\subsection{The pre-image of a given tableau for the case $k>1$}\label{section: preimage desing}\

Next, we describe the pre-image $\varphi^{-1}(\{T\})$  of a given $T\in {\rm SYT}(n-k,k)$, where $k>1$. The case $k=1$ will be treated separately in Section \ref{section case k=1}. The case $k=0$, where the SYT has only one row, was already treated in Claim \ref{claim: partitions without descents}, as by Rule (0), set partitions $\pi$ with ${\rm desing}(\pi)=0$ are mapped by $\varphi$ to standard Young tableaux with one row.

\medskip

As we shall see, the pre-image of $T$ is composed of several subsets.

\subsubsection{The first subset $L_{n-k}$}

\begin{definition}\label{def_Lnk}
Denote the elements in the first row of $T$ by $R_1$
and the elements in its second row by $R_2$.
Let $$L_{n-k}=L_{n-k}(T)=\left\{\pi \in \mathcal{PS}et(n) \mid {\rm Sing}(\pi)=R_1\right\}.$$
\end{definition}

We claim:
\begin{claim}\label{L(n-k) contained in preimage}
    $L_{n-k}\subseteq\varphi^{-1}(\{T\})$.
\end{claim}

\begin{proof}
Let $\pi\in L_{n-k}$ and we have to show that $\varphi(\pi)=T$, i.e. the first row of $\varphi(\pi)$ coincides with ${\rm Sing}(\pi)$. Let $x \in [n]$.

If $x$ appears in the first row of $T$, then by definition $x\in {\rm Sing}(\pi)$, and hence $x$ is in the first row of $\varphi(\pi)$ by Rule (1).

Now assume that $x$ is in the second row of $T$, hence $x\notin {\rm Sing}(\pi)$. Then we have several options:
\begin{enumerate}
\item If $x \in {\rm SucDes}(\pi)$, then by Rule (1), $x$ appears in the second row of $\varphi(\pi)$ as well.

\medskip

\item If $1 \leq x < d_1$, then there exists a minimal number $\ell \geq 1$ such that $x-\ell$ is located in the first row of $T$.
This implies by definition of $\pi\in L_{n-k}$, that $x-\ell\in {\rm Sing}(\pi)$, but $x-\ell+1\notin {\rm Sing}(\pi)$, so that $x-\ell\in {\rm Desing}(\pi)$. Since $x-\ell < x < d_1$, this contradicts the assumption that $d_1$ is the minimal descent of $\pi$, so this case is impossible.

\medskip

\item Otherwise, $x \in H(\pi)$. Assume to the contrary that $x$ is not contained in the second row of $\varphi(\pi)$. Moreover, assume that it is the smallest element of $H(\pi)$ with this property, thus $m_x=0$, but $m_h>0$ for all $h\in H(\pi)$ satisfying $h<x$. According to this assumption, there are $p \geq 0$ elements in $H(\pi)$ smaller than $x$, which are located in the second row of $\varphi(\pi)$ by Rule (2), appearing also in the second row of $T$. The element $x$ is the smallest element of $H(\pi)$ that appears in the second row of $T$, but not in the second row of $\varphi(\pi)$.
Note also that by Rule (1), there are $v>0$ elements from the set ${\rm SucDes}(\pi)$ which are smaller than $x$ located in the second row of both tableaux $T$ and $\varphi(\pi)$.

Assume that $x$ is located in place $(2,i+1)$ of the tableau $T$, where $i=p+v$, thus by the fact that $T$ is standard, we have $x\geq 2i+2$.
Each one of the $i$
elements preceding $x$ in the second row of both $T$ and $\varphi(\pi)$ must have a corresponding element in the first row of $\varphi(\pi)$, since:

\medskip

\begin{itemize}
\item Each element of ${\rm SucDes}(\pi)$ has a corresponding descent element in the first row.

\medskip

\item Each element  $h\in H(\pi)$ satisfying $h<x$, which is located in the second row of $\varphi(\pi)$ has a mate singleton in the first row of $\varphi(\pi)$ due to the location condition (see Remark \ref{rem-pairing}).
\end{itemize}

\medskip

This gives us $i$ elements in the first row of both tableaux $T$ and $\varphi(\pi)$ which are smaller than $x$, and in all, $2i$ elements in both rows of both tableaux $T$ and $\varphi(\pi)$, which
are smaller than $x$: $i$ elements in each row. On the other hand, note that there are $i+1$ elements which are smaller than $x$ in the first row of the tableau $T$ (as the element in place $(1,i+1)$ is also smaller than $x$), so by the construction of  $\pi\in L_{n-k}$ from $T$, there is at least one singleton $\{y\}$ in $\pi$, satisfying $y<x$, which was not counted among the $i$ elements from the first row. This singleton $\{y\}$ is not a descent (as we counted them all in $v$) and is unmatched by any element in the second row of $\varphi(\pi)$, forcing $m_x>0$, a contradiction.

The idea of this subtle proof, can be illustrated by reconsidering Example \ref{example partition to table}(1) with  $x=7$, $i=2$ and $y=4$.
\end{enumerate}

\vspace{-22pt}\end{proof}

\medskip

Note that $|L_{n-k}|={\rm Bell}_{\geq2}(k)$, since there are ${\rm Bell}_{\geq2}(k)$ possibilities to divide $k$ elements (appearing in the second row of $T$) in blocks of size at least $2$.

\subsubsection{Removable elements}

For exploring the other subsets in the pre-image of a given tableau $T$, we introduce the notion of a {\it removable element}:
\begin{definition}
Let $T \in {\rm SYT}(n-k,k)$  and let $\pi\in L_{n-k}(T)$. While scanning the set ${\rm Sing}(\pi) \, \backslash \, {\rm Desing}(\pi)$  increasingly, an element $a \in {\rm Sing}(\pi)\, \backslash \, {\rm Desing}(\pi)$ is called {\em removable} if for every $h\in H(\pi)$ satisfying $h>a$, we have that the removal of the singleton $a$ from ${\rm Sing}(\pi)$, together with all removable elements which are smaller than $a$ if they exist, does not affect the positivity of $m_h$.

Denote by ${\rm Rem}(T)$ the set of all removable elements of $T$.
\end{definition}

Note that after discovering that an element $a$ is not removable, we can actually jump to the first singleton appearing {\it after} the descent following $a$.

Note also that the definition of $
{\rm Rem}(T)$ is independent of the choice of $\pi\in L_{n-k}$. Note also that $
{\rm Rem}(T) \subsetneqq R_1$.

\begin{example}\label{example -removable elements}
Given:
$$T={\begin{ytableau} 1&2&3&6&7&8&12&13 \\ 4&5&9&10&11&14 \end{ytableau}} \in {\rm SYT} (8,6),$$
we choose:

$$\pi=\{\{1\}, \{2\},\{3\}, \{4,5,9,10,11,14\},\{6\},\{ 7\}, \{8\}, \{12\} ,\{ 13\}\}\in L_{14-6}.$$


We have
$${\rm Sing}(\pi)=\{1,2,3,6,7,8,12,13\},\ {\rm Desing}(\pi)=\{3,8,13\},$$
$${\rm SucDes}(\pi)=\{4,9,14\},$$
which implies: $H(\pi)=\{5,10,11\}$.
One can easily compute $m_h$ for $h\in H(\pi)$ (as we have: $t_3=2,\ t_8=2,\ t_{13}=1$):
$$m_5 =2-0=2>0, \ m_{10}=2+2-1=3>0, \ m_{11}=2+2-2=2>0.$$

In order to figure out the removable elements of $T$, we scan increasingly the set $${\rm Sing}(\pi) \, \backslash \, {\rm Desing}(\pi)=\{1,2,6,7,12\}:$$
\begin{itemize}

\item The element $1$ is removable, as after its deletion from the set of singletons, we have (as $t'_3=1$):
$m'_5 =1-0=1>0, \ m'_{10}=1+2-1=2>0,\break m'_{11}=1+2-2=1>0$, so its deletion does not affect the positivity of $m_h$ for every $h$.

\medskip

\item The element $2$ is not removable,  as after its deletion from the set of singletons (together with $1$), we have (as $t_3''=0$):
$m''_5=0-0=0$,
and hence the deletion of $2$ affects the positivity of $m_5$.

\medskip

\item The singleton $6$ is also not removable, since after its deletion from the set of singletons (together with $1$), we have:
$m'''_{11}=1+1-2=0$.

\medskip

\item We jump to the singleton $12$ which is removable, as there is no $h \in H(\pi)$ satisfying $h>12$.
\end{itemize}
Therefore, we have: ${\rm Rem}(T)=\{1,12\}$.
\end{example}

\begin{remark}\label{rem:motivation-removable}
The removable elements have a nice visual interpretation in the stack model (see Remark \ref{rem-pairing}) as the elements that are left inside the stack after locating all the elements of $H(\pi)$. For example, if we apply the stack model for Example \ref{example -removable elements}, we have that the elements $1,12$ are indeed left inside the stack at the end of the process, see Figure \ref{fig:stack example2}.

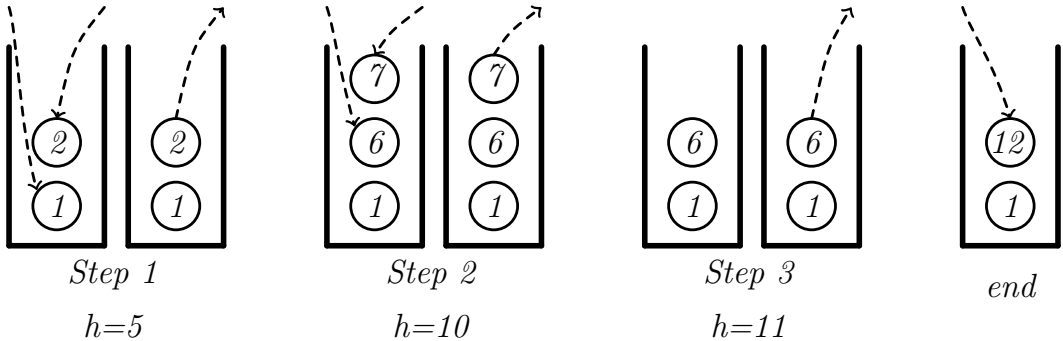
\begin{figure}[H]
    \centering
\resizebox{0.9\textwidth}{!}{\begin{tikzpicture}[cap=round,line width=2pt]

\foreach \i in
    {1, 2, 3}
  {
    \draw[-] (4*\i-2.1,0) -- (1.2+4*\i-2.1,0);
    \draw[-] (4*\i-2.1,0) -- (0+4*\i-2.1,2.5);
    \draw[-] (4*\i-0.9,0) -- (4*\i-0.9,2.5);
    \node[below] at (0.7+4*\i-3,0){Step \i};
  }

\foreach \i in
    {1, 2, 3, 4}
  {
    \draw[-] (4*\i-3.6,0) -- (1.2+4*\i-3.6,0);
    \draw[-] (4*\i-3.6,0) -- (0+4*\i-3.6,2.5);
    \draw[-] (4*\i-2.4,0) -- (4*\i-2.4,2.5);
  }

\draw[line width=1pt]  (1,0.5) circle (0.3cm);

\node at (1,0.5){1};

\draw[line width=1pt]  (1,1.3) circle (0.3cm);

\node at (1,1.3){2};

\draw[line width=1pt,dashed,->] (0.4,3) .. controls (0.5,2.7) and (0.53,1.5) .. (0.75,0.7);

\draw[line width=1pt,dashed,->] (1.6,3) .. controls (1.2,2.5)  .. (1,1.6);

\draw[line width=1pt]  (2.5,0.5) circle (0.3cm);

\node at (2.5,0.5){1};

\draw[line width=1pt]  (2.5,1.3) circle (0.3cm);

\node at (2.5,1.3){2};

\draw[line width=1pt,dashed,->] (2.5,1.6) .. controls (2.7,2.5)  .. (3.1,3);

\node at (1.7,-1){h=5};

\draw[line width=1pt]  (5,0.5) circle (0.3cm);

\node at (5,0.5){1};

\draw[line width=1pt]  (5,1.3) circle (0.3cm);

\node at (5,1.3){6};

\draw[line width=1pt]  (5,2.1) circle (0.3cm);

\node at (5,2.1){7};

\draw[line width=1pt,dashed,->] (4.4,3) .. controls (4.5,2.7) and (4.53,2.3) .. (4.75,1.5);

\draw[line width=1pt,dashed,->] (5.6,3) .. controls (5.2,2.7)  .. (5,2.4);

\draw[line width=1pt]  (6.5,0.5) circle (0.3cm);

\node at (6.5,0.5){1};

\draw[line width=1pt]  (6.5,1.3) circle (0.3cm);

\node at (6.5,1.3){6};

\draw[line width=1pt]  (6.5,2.1) circle (0.3cm);

\node at (6.5,2.1){7};

\draw[line width=1pt,dashed,->] (6.5,2.4) .. controls (6.7,2.7)  .. (7.1,3);

\node at (5.7,-1){h=10};

\draw[line width=1pt]  (9,0.5) circle (0.3cm);

\node at (9,0.5){1};

\draw[line width=1pt]  (9,1.3) circle (0.3cm);

\node at (9,1.3){6};

\draw[line width=1pt]  (10.5,0.5) circle (0.3cm);

\node at (10.5,0.5){1};

\draw[line width=1pt]  (10.5,1.3) circle (0.3cm);

\node at (10.5,1.3){6};

\draw[line width=1pt,dashed,->] (10.5,1.6) .. controls (10.7,2.5)  .. (11,3);

\node at (9.7,-1){h=11};

\draw[line width=1pt]  (13,0.5) circle (0.3cm);

\node at (13,0.5){1};

\draw[line width=1pt]  (13,1.3) circle (0.3cm);

\node at (12.95,1.3){12};

\draw[line width=1pt,dashed,->] (12.4,3) .. controls (12.5,2.7) and (12.8,2.3) .. (13,1.6);

\node at (13,-0.5){end};

\end{tikzpicture}}
\caption{Applying the stack model for Example \ref{example -removable elements}}\label{fig:stack example2}
\end{figure}

Moreover, see Section \ref{visual perspectives} below for another visual way based on jeu-de-taquin to verify whether an element is removable, using the tableau $T$.
\end{remark}

\subsubsection{The other subsets in the pre-image of $T$}

After identifying the set $L_{n-k}$ as part of the pre-image of $T$, we proceed to figure out the other subsets in its pre-image. We move the minimal element $a$ of ${\rm Rem}(T)$ from the set $R_1$ to the set $R_2$, and denote the resulting sets by $R_1^{(1)}$ and $R_2^{(1)}$, respectively (see e.g. the relocation of the element $3$, which is the minimal element in ${\rm Rem}\left(\raisebox{.5ex}{\begin{smallytableau}
1 & 3 & 4 & 5 & 6 \\ 2 & 7
\end{smallytableau}}\right)= \{3,4,5\}$, from the set partition in Example \ref{example partition to table}(3) to the set partition in Example \ref{example partition to table}(4)).

\medskip

Now denote:
$$L_{n-k-1}=\left\{\pi \in \mathcal{PS}et(n)\mid {\rm Sing}(\pi)=R_1^{(1)}\right\},$$
and we claim:
\begin{claim}\label{L(n-k-1) contained in preimage}
$L_{n-k-1} \subseteq \varphi^{-1}(\{T\})$.
\end{claim}

\begin{proof}
We take set partitions $\pi \in L_{n-k}$ and $\pi' \in L_{n-k-1}$, and we have to show that $\varphi (\pi')=T$.

If $x$ appears in the first row of $T$ and $x \neq a$, then by definition $x\in {\rm Sing}(\pi')$, and hence $x$ is in the first row of $\varphi(\pi')$ by Rule (1).

Now assume that $x$ is in the second row of $T$, hence $x\notin {\rm Sing}(\pi')$. Then we have several options:
\begin{enumerate}
\item If $x \in {\rm SucDes}(\pi')$, then by Rule (1), $x$ appears in the second row of $\varphi(\pi')$ as well.

\medskip

\item The case $1 \leq x < d_1$ is impossible, similar to the proof of Claim \ref{L(n-k) contained in preimage}.

\medskip

\item Otherwise, $x \in H(\pi')$. Note that $H(\pi')=H(\pi) \cup \{a\}$. Since $a$ is a removable element, it does not affect the positivity of $m_h$ for any $h \in H(\pi)$, and thus all other elements of $H(\pi')$, except for $a$, are placed in the same row they occupy in $\varphi(\pi)=T$, and we follow the proof of Claim \ref{L(n-k) contained in preimage}.
\end{enumerate}

\medskip

It remains now to show that $a$ is indeed located in the first row of
$\varphi(\pi')$ as in $T$,
i.e. one has to show that $m_a(\pi')=0$.
It is obvious for $a<d_1$, so we can assume that $a>d_i$ for some $i\geq 1$.

Assume to the contrary that
$m_a(\pi')>0$, thus in $\pi'$ we have:
$$\sum\limits_{j=1}^i t_{d_j} > \sum\limits_{\{g\in H(\pi')\ \mid \ g<a\}}\chi(m_g>0),$$
which implies that there is some singleton $s \in {\rm Sing}(\pi)$, such that $s<a$ and $s$ was not paired to any element $h<a$ of $H(\pi')$ (see Remark \ref{rem-pairing}).

\medskip

Now we will show that $s$ is a removable element of $\pi$, leading to a contradiction to the fact that $a$ was defined to be the minimal element of ${\rm Rem}(T)$. In order to do that, we have to show that the removal of $s$ from ${\rm Sing}(\pi)$ does not affect the positivity of $m_h(\pi)$ for any $h \in H(\pi)$ satisfying $h>s$. We split our treatment into two subsets of $H(\pi)$:  $\ H(\pi) \cap [s+1,a]$ and $H(\pi) \cap [a+1,n]$:

\medskip

\begin{itemize}
\item For the subset $H(\pi) \cap [s+1,a]$, as $s$ was not paired to any element $h \in H(\pi)$ appearing in the second row of $\varphi(\pi)$ with $h<a$, any such element $h$ is still located in the second row even after the removal of $s$.

\medskip

\item For the subset $H(\pi) \cap [a+1,n]$, since $a$ is assumed to be a removable element, all the elements in $H(\pi)$ after $a$ are not affected due to its removal, and similarly will not be affected by the removal of $s$ while keeping the singleton $a$ (as is done in the process of finding the set of removable elements).
\end{itemize}
Hence we get that $s$ is a removable element of $T$, contradicting the minimality of $a$.
\end{proof}

Note also that $|L_{n-k-1}|={\rm Bell}_{\geq 2}(k+1)$, since by definition there are ${\rm Bell}_{\geq 2}(k+1)$ possibilities to divide $k+1$ elements (appearing in $R_2^{(1)}$) into blocks of size at least $2$.

\medskip

Next, move the second-to-minimal element of ${\rm Rem}(T)$ from the set $R_1^{(1)}$ to the set $R_2^{(1)}$, and denote the resulting sets by $R_1^{(2)}$ and $R_2^{(2)}$, respectively.
We denote $$L_{n-k-2}=\left\{\pi \in \mathcal{PS}et(n)\mid {\rm Sing}(\pi)=R_1^{(2)}\right\},$$
and similarly, we have that: $L_{n-k-2} \subseteq \varphi^{-1}(\{T\})$, and $|L_{n-k-2}|={\rm Bell}_{\geq 2}(k+2)$.

\medskip

We continue in this manner for all elements in ${\rm Rem}(T)$. We claim now:
\begin{prop}\label{pre-image greater than 1}
Let $T\in {\rm SYT}(n-k,k)$. Then:
\begin{equation}
\varphi^{-1}(\{T\}) = \bigcup\limits_{i=0}^{|{\rm Rem}(T)|} L_{n-k-i}.\label{equ-pre-image}
\end{equation}
Moreover, $|{\rm Rem}(T)|=n-2k$, and therefore: $|\varphi^{-1}(\{T\})| = \sum\limits_{i=k}^{n-k} {\rm Bell}_{\geq 2}(i).$
\end{prop}

\begin{proof}
First, we prove Equation (\ref{equ-pre-image}). By Claims \ref{L(n-k) contained in preimage}, \ref{L(n-k-1) contained in preimage} and the discussion afterwards, we obviously have:
$$\bigcup\limits_{i=0}^{|{\rm Rem}(T)|} L_{n-k-i} \subseteq \varphi^{-1}(\{T\}).$$
In order to prove the other direction, note that by Remark \ref{rem-pairing}, for each $\pi' \in
\varphi^{-1}(\{T\})$, one must have: $R_1 \backslash {\rm Rem}(T) \subseteq {\rm Sing}(\pi') \subseteq R_1$, as otherwise $\varphi(\pi')\neq T$.
Hence, it remains to show that the order of removing the elements of ${\rm Rem}(T)$ from the set $R_1$  must be from the left to the right as was done  above; explicitly, for each pair of elements $i,j\in {\rm Rem}(T)$ satisfying $i<j$, one has to show that $i \in {\rm Sing}(\pi')$ implies $j \in {\rm Sing}(\pi')$ (and therefore $\pi' \in L_{n-k-p}$ for some $0\leq p\leq |{\rm Rem} (T)|$). Indeed,  assume to the contrary that $i \in {\rm Sing}(\pi')$ but $j \notin {\rm Sing}(\pi')$. Note that $j>d_1$, otherwise it creates a new descent. Therefore, we have that $j \in H(\pi')$, and by Rule (2), $j$ should be located in the second row of $\varphi(\pi')$ (as the existence of the singleton $\{i\}$ together with the existence of the paired elements in the first row smaller than $j$, in the sense of Remark \ref{rem-pairing}, imply the positivity of $m_j$), so that $\pi' \notin \varphi^{-1}(\{T\})$, a contradiction.

\medskip

Next, we have to prove that $|{\rm Rem}(T)|=n-2k$, or alternatively, that the sum of ${\rm desing}(\pi)$ and the number of non-removable elements is $k$.
Indeed, note that the elements of the second row of $T$ comprise the set ${\rm SucDes}(\pi)$ and the elements of $H(\pi)$ which are placed in the second row during the application of Rule (2), due to the existence of paired elements in ${\rm Sing}(\pi)$.
Every such element of ${\rm Sing}(\pi)$ is by definition non-removable, so that the number of removable elements is equal to the difference $n-2k$ between the two rows of $T$.

Therefore, in this case, the above union is $\bigcup\limits_{j=k}^{n-k} L_{j}$, and hence its size is  $\sum\limits_{j=k}^{n-k} {\rm Bell}_{\geq 2}(j)$.
\end{proof}

\medskip

The following example illustrates the procedure to find the pre-image of a given standard Young tableau for $n=8$ and $k=3$:
\begin{example}\label{example procedure k geq 2}
Let $T={\begin{ytableau} 1&3&4&7&8 \\ 2&5&6 \end{ytableau}} \ .$

In the first step, we have $R_1=\{1,3,4,7,8\}$ and $R_2=\{2,5,6\}$, and so we get the single set partition:
$$\pi_1=1 \mid 256 \mid 3 \mid 4 \mid 7 \mid 8 \in L_{8-3},$$ as there is only ${\rm Bell}_{\geq 2}(3)=1$ way to divide the set $\{2,5,6\}$ into blocks of size at least $2$.

\medskip

Note that in this case we have $H(\pi)=\{6\}$ with $m_6=1>0$,  so that ${\rm Rem}(T)=\{7,8\}$. We first move $7$ from $R_1$ to $R_2$ to get $R_1^{(1)}=\{1,3,4,8\}$ and $R_2^{(1)}=\{2,5,6,7\}$, which contributes the following ${\rm Bell}_{\geq 2}(4)=4$ set partitions in $L_{8-3-1}$:
$$\pi_2= 1 \mid 2567 \mid 3 \mid 4 \mid 8, \qquad \pi_3=1\mid 25 \mid  3 \mid 4 \mid 67 \mid 8,$$
$$\pi_4=1\mid 26 \mid  3 \mid 4 \mid 57 \mid 8, \qquad \pi_5= 1\mid 27 \mid  3 \mid 4 \mid 56 \mid 8.$$
In the last step, we move $8$ from $R_1^{(1)}$ to $R_2^{(1)}$ to get $R_1^{(2)}=\{1,3,4\}$ and\break $R_2^{(2)}=\{2,5,6,7,8\}$, which contributes ${\rm Bell}_{\geq 2}(5)=11$ set partitions $\{\pi_6,\dots,\pi_{16}\}$ in $L_{8-3-2}$, where $\{1\},\{3\},\{4\}$ are singletons and the elements $\{2,5,6,7,8\}$ have to be divided into blocks of size at least $2$.    All in all, the pre-image of $T$ is the set $\{\pi_1,\dots,\pi_{16}\}$ of size: $${\rm Bell}_{\geq2}(3)+{\rm Bell}_{\geq2}(4)+{\rm Bell}_{\geq2}(5)=16.$$
\end{example}

\subsection{The case $k=1$}\label{section case k=1}

Now, we deal with the case $k=1$, i.e. let $T$ be a standard Young tableau of shape $(n-1,1)$. We must have $|{\rm Des}(T)|=|\{d\}|=1$ for some $1 \leq d <n$. We choose a set partition $\pi \in \varphi^{-1}(\{T\})$ with a maximal number of singletons, depending on the value of $d$:
$$\pi =\left\{ \begin{array}{lll}
\vspace{7pt}\{\{1\},\{2,3\}\} \cup \{\{i\}\mid 4\leq i \leq n\}  & \ & d=1 \\
\{\{1,d+1\}\} \cup \{\{i\}\mid 2\leq i \leq n, i \neq d+1\}  & \ & 1<d<n
\end{array} \right. $$
The reason for choosing such a set partition $\pi$ is due to the requirement that a singleton $\{d\}$ is an element of ${\rm Desing}(\pi)$ only if the element $d+1$ {\it is not a singleton}, so we join into one block the element $d+1$ and the minimal possible element, which is $1$ in the case that $d\neq 1$, and $3$  otherwise.

\medskip

Define $L_{n-2}=\{\pi\}$ (in this case, we use the index $n-2$, as there are $n-2$ singletons in $\pi$), and consequently, we form the set $R_1$ to be the singletons of $\pi$, and $R_2=[n]\setminus R_1$, as follows.
$$R_1 ={\rm Sing}(\pi) = \left\{ \begin{array}{ll}
\vspace{5pt}\, [n] \, \backslash \, \{2,3\}  & \ d=1 \\
\, [n] \, \backslash \, \{1,d+1\}  & \ 1<d<n
\end{array} \right.
\mbox{and }
R_2 =\left\{ \begin{array}{ll}
\vspace{5pt} \{2,3\}  & \ d=1 \\
\{1,d+1\}  & \ 1<d<n
\end{array} \right. $$
Hence:
$H(\pi) =\left\{ \begin{array}{lll}
\vspace{5pt}\{3\}  & \ \ \ & d=1 \\
\emptyset  & \ \ \ & 1<d<n
\end{array} \right. $

Note that in the case that $d=1$, we have $m_3=0$. Hence, in both cases, all the elements except for $d$ are removable.
Now we continue in the manner of the proof of the case $k>1$ to get the following result.
\begin{prop}\label{pre-image equals 1}
Let $T\in {\rm SYT}(n-1,1)$ and let $L_i$ be defined as above. Then:
$$\varphi^{-1}(\{T\}) = \bigcup\limits_{i=1}^{n-2} L_i.$$
Therefore: $|\varphi^{-1}(\{T\})| =\sum\limits_{i=2}^{n-1} {\rm Bell}_{\geq 2}(i).$
\end{prop}

\medskip

We give an example for the case $k=1$:
\begin{example}
Let $T={\begin{ytableau} 1&2&4&5&6 \\ 3 \end{ytableau}} \ .$

In the first step, since $d=2>1$, we take the single set partition $$\pi_1=13 \mid 2 \mid 4 \mid 5 \mid 6.$$
Hence, we have: $R_1=\{2,4,5,6\}$ and $R_2=\{1,3\}$.
Note that in this case we have $H(\pi_1)=\emptyset$, so that ${\rm Rem}(T)=\{1,4,5,6\}$.

We first move the element $4$ from $R_1$ to $R_2$ to get $R_1^{(1)}=\{2,5,6\}$ and $R_2^{(1)}=\{1,3,4\}$, which contributes the ${\rm Bell}_{\geq 2}(3)=1$ set partition
$$\pi_2=134 \mid 2 \mid 5 \mid 6.$$
Next, we move the element $5$ from $R_1^{(1)}$ to $R_2^{(1)}$ to get $R_1^{(2)}=\{2,6\}$ and\break $R_2^{(2)}=\{1,3, 4,5\}$, which contributes the following ${\rm Bell}_{\geq 2}(4)=4$ set partitions:
$$\pi_3=1345 \mid 2 \mid 6, \qquad \pi_4=13 \mid 2 \mid 45 \mid 6,$$
$$\pi_5=14 \mid 2 \mid 35 \mid 6,\qquad \pi_6=15 \mid 2 \mid 34 \mid 6.$$
In the last step, we move the element $6$ from $R_1^{(2)}$ to $R_2^{(2)}$ to get $R_1^{(3)}=\{2\}$ and $R_2^{(3)}=\{1,3,4,5,6\}$, which contributes ${\rm Bell}_{\geq 2}(5)=11$ set partitions $\{\pi_7,\dots,\pi_{17}\}$, where $\{2\}$ is the unique singleton and the elements $\{1,3,4,5,6\}$ are divided into blocks of size at least $2$.    All in all, the pre-image of $T$ is the set $\{\pi_1,\dots,\pi_{17}\}$ of size: $${\rm Bell}_{\geq2}(2)+{\rm Bell}_{\geq2}(3)+{\rm Bell}_{\geq2}(4)+{\rm Bell}_{\geq2}(5)=17.$$
\end{example}

Claims \ref{claim: partitions without descents} and \ref{indeed standard}, together with Propositions
\ref{pre-image greater than 1} and \ref{pre-image equals 1}, imply Theorem \ref{main theorem}.

\begin{remark}\label{connection - Marmor for all partitions}
As mentioned above (see after Example \ref{first example}), the statistic ${\rm Desing}$ is sparse, and based on our analysis in Theorem \ref{main theorem} above, it can be proven that this statistic satisfies the conditions of Theorem 1.8 of Marmor \cite{Marmor}. Hence, our work on this statistic can be considered as an explicit combinatorial application of Marmor's general result, adding to it an explicit bijection.
Note that the coefficients $\left|\{\pi\in \mathcal{PS}et(n) \ : \ {\rm Desing}(\pi)=\{1,3,\dots, 2k-1\}\}\right|$ appearing at the end of the formulation of Marmor's Theorem 1.8 \cite{Marmor}
coincide with our coefficient $e_{n-k,k}$, as the single Young tableau $T$ with two rows having ${\rm Des}(T)=\{1,3,\dots, 2k-1\}$ is:
$$\begin{ytableau}
    1 & 3 & 5 & \cdots &  {\scriptscriptstyle 2k-1} & {\scriptscriptstyle 2k+1} & {\scriptscriptstyle 2k+2} & \cdots &n \\
    2 & 4 & 6 & \cdots &  2k
\end{ytableau}\ , $$
and $|\varphi^{-1}(\{T\})|=e_{n-k,k}$.
\end{remark}

\subsection{Visual perspectives}\label{visual perspectives}

Two issues above can be presented nicely in a visual way: the removable elements and the table of coefficients $e_{n-k,k}$.

\subsubsection{Visual recognition of a removable element}
We present a nice way to
verify whether an element $a$ is removable, given a tableau $T$, where
$T=\varphi(\pi)$ for some set partition $\pi$.
After deleting an element $a$ from the first row of $T$, together with all removable elements to its left if they exist, standardizing the remaining numbers in the tableau and applying the `jeu de taquin' process on the first row of $T$, i.e. rectifying $T$, we should get a valid tableau $T' \in {\rm SYT}(n-k-y-1,k)$ for some $y\geq0$.

The equivalence between this visual way and the definition of a removable element relies on the idea that an element cannot be removable if there is an element in the second row that `needs' that element from the first row, see Remark \ref{rem-pairing}.

\medskip

\begin{example}
We illustrate this visual way by few examples:

\medskip

(1) Given the tableau $T={\begin{smallytableau} 1&2&3&6&7 \\ 4&5 \end{smallytableau}}$, take $\pi= 1\mid 2 \mid 3 \mid 45 \mid 6 \mid 7 \in \varphi^{-1}(\{T\})$, and thus:
$${\rm Sing}(\pi)=\{1,2,3,6,7\}, \ {\rm Desing}(\pi)=\{3\}, \ {\rm SucDes}(\pi)=\{4\}.$$
Hence: $H(\pi)=\{5\}$. Using the stack model, after inserting $1$ and then $2$ as the elements of the truncated run of $3$, we pop out $2$ due to $5$. Afterwards, we insert to the  stack the elements $6,7$ which appear after the last element of ${\rm Desing}(\pi)$. The elements $1,6,7$ are left in the stack, and hence they are  removable.
On the other hand, $2$ is not removable, as it was already popped out.

Alternatively, after the deletion of $1$ from the tableau $T$, we get: $\begin{smallytableau} \none&2&3&6&7 \\ 4&5 \end{smallytableau}$, and then after standardizing and rectifying $T$, we have:
$\begin{smallytableau} 1&2&5&6 \\ 3&4 \end{smallytableau}$, which is a valid SYT.

On the other hand, $2$ is not a removable element, as after the deletion of $2$ (together with $1$) from $T$, we get: $\begin{smallytableau} \none&\none&3&6&7 \\ 4&5 \end{smallytableau}$, and then after standardizing and rectifying $T$, we have:
${\begin{smallytableau} 1&4&5 \\ 2&3 \end{smallytableau}}$, which is not a valid SYT.

Moreover, the elements $6,7$ are both removable elements, as after their deletion, together with the deletion of the previous removable element $1$, we get: $\begin{smallytableau} \none&2&3 \\ 4&5 \end{smallytableau}$, and then after standardizing and rectifying $T$, we obtain the valid SYT \
${\begin{smallytableau} 1&2 \\ 3&4 \end{smallytableau}}$.

\medskip

(2) Given the tableau $T={\begin{smallytableau} 1&2&4&5 \\ 3&6&7 \end{smallytableau}}$, note that $1$ is a removable element,
as after its deletion, we get: $\begin{smallytableau} \none&2&4&5 \\ 3&6&7 \end{smallytableau}$, and then after standardizing and rectifying $T$, we have the valid SYT \
${\begin{smallytableau} 1&3&4 \\ 2&5&6 \end{smallytableau}}$.

Note that on the face of it,
the element $4$ could be wrongly considered as a removable element as its deletion yields a valid SYT, but as we require that we read the elements of the first row of $T$ from minimum to maximum, we first explore that $1$ is a removable element, and then the condition on $4$ to be a removable element does not hold anymore, since it forces also the deletion of $1$ which together with the deletion of $4$ yields a non-valid SYT.

\end{example}

\medskip

\subsubsection{A table presentation of the coefficients $e_{n-k,k}$}\label{visual 2}
The second visual issue is organizing the coefficients $e_{n-k,k}$ in a table and exploring their mutual properties. In Table \ref{tab:s_nk}, we present some numerical values of these coefficients, and we attach to them numbers from the Catalan's triangle, denoted $c_{n-k,k}$, which count the number of SYT of a given shape (see details below).

\begin{table}[H]
\centering
{\tiny$$\begin{array}{|c|c:c:c:c:c:c:c|}
\hline
n & e_{n,0} \ [c_{n,0}] & e_{n-1,1} \ [c_{n-1,1}] & e_{n-2,2} \ [c_{n-2,2}] & e_{n-3,3} \ [c_{n-3,3}] & e_{n-4,4} \ [c_{n-4,4}] &  e_{n-5,5} \ [c_{n-5,5}] & e_{n-6,6} \ [c_{n-6,6}] \\
\hline
\raisebox{-.5ex}{0}  &  \raisebox{-.5ex}{\fbox{\bf 1}\ [1]} &&&&&& \\
1  &  1\ [1]&&&&&& \\
2  &  2\ [1] & \fbox{\bf 0} \ [1] &&&&& \\
3  &  3\ [1] & 1\ [2]&&&&& \\
4  &  7\ [1]  & 2\ [3] & \fbox{\bf 1}\ [2] &&&& \\
5  &  18\ [1] & 6\ [4] & 2\ [5] &&&& \\
6  &  59\ [1] & 17\ [5] & 6\ [9]  & \fbox{\bf 1}\ [5] &&& \\
7  &  221\ [1] & 58\ [6] & 17\ [14] & 5\  [14]&&& \\
8  &  936\ [1]  & 220\ [7] & 58\ [20]  & 16\ [28] & \fbox{\bf 4}\ [14]  && \\
9  &  4361\ [1] & 935\ [8] & 220\ [27] & 57\ [48] & 15\ [42] && \\
10 &  22083\ [1]  & 4360\ [9]  & 935\ [35] & 219\ [75] & 56\ [90] & \fbox{\bf 11}\ [42] &\\
11 &  120336\ [1] & 22082\ [10] & 4360\ [44] & 934\ [110] & 218\ [165] & 52 \ [132] & \\
12 & 700653\ [1] & 120335\ [11] & 22082\ [54] & 4359\ [154] & 933\ [275] & 214\ [297] & \raisebox{.1ex}{\fbox{\bf 41}\ [132]} \\
\hline
\end{array}$$}
\caption{Numerical values of $e_{n-k,k}$ and $c_{n-k,k}$ for small values of $n,k$.
}
\label{tab:s_nk}
\end{table}

The first column of the table is the sequence $\{e_{n,0}\}$ whose $n$th element counts the number of set partitions of up to $n$ elements without singletons (which is the sum $\sum\limits_{i=0}^n {\rm Bell}_{\geq2}(i)$ in Claim \ref{claim: partitions without descents} above).

The (knight-steps) sequence of elements in the table which are squared, is the sequence\break $\{e_{n,n}\}=\{{\rm Bell}_{\geq 2}(n)\}$, which counts the set partitions of $[n]$ without singletons.

By the results above (see e.g. Theorem \ref{main theorem}), there is a nice visual way to compute the coefficient $e_{n-k,k}$  by the elements appearing in the squares, using the table: start from the location of the element $e_{n-k,k}$ that you want to compute and go in two directions: one direction is up-wise and the other is diagonally right and down. In both directions, you reach squared elements. Now, $e_{n-k,k}$ is the sum of all the squared elements between the two squared elements you have reached (including them).

For example, $e_{5,2}=e_{7-2,2}=17$, since if we go in the above-mentioned directions, we find the squared elements $e_{2,2}=1$ and $e_{5,5}=11$, and the sum of the squared elements between them is indeed:
$$e_{2,2}+e_{3,3}+e_{4,4}+e_{5,5}={\rm Bell}_{\geq 2}(2)+{\rm Bell}_{\geq 2}(3)+{\rm Bell}_{\geq 2}(4)+{\rm Bell}_{\geq 2}(5)=1+1+4+11=17.$$
As a consequence, if one looks on the right-down diagonals of this table, the sequence of differences of the these diagonal sequences is exactly the sequence $\{e_{n,n}\}=\{{\rm Bell}_{\geq 2}(n)\}$ whose elements are squared.

As mentioned above, the numbers in brackets are elements from the Catalan's triangle (also called {\it ballot numbers}, OEIS sequence: A009766, see \cite{OEIS}), which satisfy the recursion $c_{n,k}=c_{n-1,k}+c_{n,k-1}$ with the boundary condition $c_{0,0}=1$. They are connected to Catalan numbers via the following equation: $\sum\limits_{k=0}^n c_{n,k} = C(n+1)$, where $C(n+1)$ is the $(n+1)$th Catalan number. Their explicit formula is: $c_{n, k} = \frac{n-k+1}{n+1}{n+k \choose n}$. By the Hook length formula (see e.g. \cite[Corollary 7.21.6]{EC2}), the element $c_{n-k,k}$ from the Catalan's triangle counts the number of standard Young tableaux of shape $(n-k,k)$.

\medskip

Note that Corollary \ref{coro:bell sum} can be read out of the table as a weighted sum of its rows.

\section{A special case: the Schur-positivity of non-crossing partitions with respect to the parameter desing}\label{section: desing for noncrossing partitions}
In this section, we discuss a special case, which deals with the Schur-positivity of the set $\mathcal{NC}(n)$ of the non-crossing partitions in $\mathcal{PS}et(n)$, with respect to the parameter ${\rm desing}$.

\begin{definition}\label{definition Riordan numbers}
Denote by $R_n$ the set of non-crossing set partitions of $[n]$ without singletons and let $r_n=|R_n|$ for each $n$.
Ira Gessel suggested to call the elements of this sequence $\{r_n\}$ {\em Riordan numbers}, which is also the sequence A005043 in the OEIS \cite{OEIS}:
$$1,0,1,1,3, 6, 15, 36, 91, 232, \dots$$
\end{definition}

\medskip

We start with the following observation:
\begin{claim}\label{claim: nc partitions without descents}
The number of non-crossing set partitions $\pi  $ of $[n]$ having ${\rm desing}(\pi)=0$ is $\beta_n:=\sum\limits_{i=0}^n r_i$.
\end{claim}
The proof is similar to the proof of Claim \ref{claim: partitions without descents} above.

\medskip
The sequence $\left\{\beta_n\right\}$ appears as sequence
A082395 in the OEIS \cite{OEIS}:
$$1, 1, 2, 3, 6, 12, 27, 63, 154, 386,\dots,$$
which counts the shifted Young tableaux for $\lambda \vdash (n+1)$ of height not greater than $3$ (where the change from $n$ to $n+1$ is due to an offset of values in one place between our sequence $\left\{\beta_n\right\}$ and sequence A082395). Shifted Young tableaux of height not greater than $3$ are Young tableaux having up to $3$ rows, where each row is shifted one place to the right (see Hassani \cite{HassaniThesis}; for the  definition of shifted Young tableaux, see e.g. Sagan \cite{SaganShifted}).

\begin{remark}
It is known that the number of (ordinary) standard Young tableaux of $n$ elements having at most three rows is equal to the number of non-crossing involutions in $S_n$ (see \cite[Corollary 14.4.18]{BonaEnumerCombin}), and are both counted by the Motzkin number $M_n$ (see Section \ref{subsec:motzkin numbers} above).
Eu \cite{Eu} provides a simple recursive bijection between Motzkin paths of length $n$ and
the aforementioned standard Young tableaux.
\end{remark}

In light of this remark, the following question is now natural:
\begin{question}\label{question generalize Eu}
As the number of non-crossing partitions $\pi$ satisfying ${\rm desing}(\pi)=0$ coincides with the number of shifted Young tableaux of $n+1$ elements having at most $3$ rows by Claim \ref{claim: nc partitions without descents} above, can Eu's bijection be generalized to the case of shifted Young tableaux?
\end{question}

\medskip

The following theorem presents the Schur-positivity of the set of non-crossing set partitions of $[n]$ with respect to the parameter ${\rm Desing}$:
\begin{thm}\label{main theorem noncrossing}
Define for $n\geq0$:
$$d_{n-k,k} = \left\{ \begin{array}{ll}
\sum\limits_{i=2}^{n-1} r_i & k=1\\
\sum\limits_{i=k}^{n-k} r_i & (k=0) \mbox{ or } (k>1) \\
\end{array}\right. .$$
Then we have:
\begin{equation}
\sum\limits_{\pi \in {\mathcal NC}(n)}{\mathbf q}^{{\rm Desing}(\pi)}=\sum\limits_{k=0}^{\left\lfloor \frac{n}{2}\right\rfloor} d_{n-k,k}
\left(\sum\limits_{T \in {\rm SYT}(n-k,k)}{\mathbf q}^{{\rm Des}(T)}\right).\end{equation}

We conclude, using Gessel's theorem (Theorem \ref{gessel} above), that
\begin{equation}
    \mathcal{Q}_n(\mathcal{NC}(n))=\sum\limits_{k=0}^{\left\lfloor\frac{n}{2}\right\rfloor} d_{n-k,k} s_{(n-k,k)}.
    \end{equation}
\end{thm}

The proof of Theorem \ref{main theorem noncrossing} goes verbatim as the proof of Theorem \ref{main theorem}, where we replace $\mathcal{PS}et(n)$ by $\mathcal{NC}(n)$ and ${\rm Bell}_{\geq2}(i)$ by $r_i$.

\begin{remark}
Similar to Remark \ref{connection - Marmor for all partitions}, also in the non-crossing case, our work on this statistic can be considered as another explicit combinatorial application of Marmor's general result in the non-crossing case, adding to it an explicit bijection.  Likewise, note that the coefficients $\left|\{\pi\in \mathcal{NC}(n) \ : \ {\rm Desing}(\pi)=\{1,3,\dots, 2k-1\}\}\right|$ appearing at the end of the formulation of Marmor's Theorem 1.8 \cite{Marmor}
coincide with our coefficient $d_{n-k,k}$, as the single Young tableau $T$ with two rows having ${\rm Des}(T)=\{1,3,\dots, 2k-1\}$ is:
$$\begin{ytableau}
    1 & 3 & 5 & \cdots &  {\scriptscriptstyle 2k-1} & {\scriptscriptstyle 2k+1} & {\scriptscriptstyle 2k+2} & \cdots &n \\
    2 & 4 & 6 & \cdots &  2k
\end{ytableau}\ , $$
and $|\varphi^{-1}(\{T\})|=d_{n-k,k}$.
\end{remark}

\medskip

\begin{remark}
(a) The sequence $\{d_{n,0}\}$ counts the number of non-crossing set partitions of up to $n$ elements $\pi$ with $desing(\pi)=0$ (see Claim \ref{claim: nc partitions without descents} above). The sequence $\{d_{n,n}\}=\{r_n\}$ counts the non-crossing set partitions of $[n]$ without singletons at all.

\medskip
(b) The total number of non-crossing set partitions of $n$ elements (counted by the $(n+1)$th Catalan number) can be also computed by: $\sum\limits_{k=0}^{\left\lfloor\frac{n}{2}\right\rfloor} d_{n-k,k}c_{n-k,k}$, where the summand $d_{n-k,k}c_{n-k,k}$ counts the number of non-crossing set partitions which correspond by our bijection to standard Young tableaux of shape $(n-k,k)$. We have found no previous mention of this equality in the literature.

\end{remark}


\section{A refinement: the Schur-positivity of set partitions with a given number of blocks}\label{section:refinement}

When refining the analysis of Section \ref{subsection schur positivity with respet to desing} by fixing the number of blocks, the results are very similar, though, interestingly enough, some of the coefficients turn out to be Eulerian numbers.

Recall the definition of  $\mathcal{PS}et(n,b)$ as the set of set partitions of $[n]$ having $b$ blocks from the beginning of Section \ref{section on schur positivity for des sharing block}.
Recall also the definition of the {\it associated
Stirling numbers of the second kind}:
\begin{definition}\label{definition associated Stirling number}
Let $\{S_{\geq 2}(n,b)\}$ be the sequence of {\em associated Stirling numbers of the second kind}  or the {\em Ward numbers} (in a different parametrization, expressed by writing the diagonals of Table \ref{table A008299} as rows, see \cite{Wa} and sequence A134991 in OEIS \cite{OEIS}), which count the number of set partitions of the set $[n]$ into $b$ blocks, each of size at least $2$. This is sequence A008299 of OEIS \cite{OEIS}, and arranged as a triangle, its first rows are presented in Table \ref{table A008299}.

\begin{table}[H]
$${\small\begin{NiceArray}{|w{c}{0.4cm}|c;c;c;c;c|}
\hline
\Gape[0.35cm][0.35cm]{\diagbox{\,\,n}{b\,}} & 1 & 2 & 3 & 4 & 5\\
\hline
2  &  1  &&&& \\
3  &  1  &&&&\\
4  &  1 & 3 &&&\\
5  &  1 & 10 &&&\\
6  &  1 & 25 & 15 &&  \\
7  &  1 & 56 & 105 &&  \\
8  &  1 & 119 & 490 & 105 & \\
9  &  1 & 246 & 1918 & 1260 & \\
10 &  1 & 501 & 6825 & 9450 & 945 \\
\hline
\end{NiceArray}}$$
\caption{Table form of the associated Stirling numbers of the second kind (sequence A008299 in OEIS)}\label{table A008299}
\end{table}
\end{definition}

Note that the sequence of sum of rows of $\left\{S_{\geq 2}(n,b)\right\}$ is the sequence $\left\{{\rm Bell}_{\geq 2}(n)\right\}$ which played a central role in Section \ref{subsection schur positivity with respet to desing}.

\begin{definition}\label{def: e_{n-k,k}^{(b)}}
Define for $n\geq0,0 \leq k\leq \lfloor\frac{n}{2} \rfloor$ and $1 \leq b \leq n$,

$$e^{(b)}_{n-k,k} = \left\{ \begin{array}{ll}
\sum\limits_{t=1}^{n-2} S_{\geq 2}(n-t,b-t) & \ \ k=1\\
\sum\limits_{t=k}^{n-k} S_{\geq 2}(n-t,b-t) & \ \ (k=0) \mbox{ or } (k>1). \\
\end{array}\right.$$
\end{definition}

Then we have:

\begin{thm}\label{main theorem blocks}

\begin{equation}\label{eqn main theorem blocks}
\sum\limits_{\pi \in {\mathcal PS}et(n,b)}{\mathbf q}^{{\rm Desing}(\pi)}=\sum\limits_{k=0}^{\left\lfloor \frac{n}{2}\right\rfloor} e^{(b)}_{n-k,k}
\left(\sum\limits_{T \in {\rm SYT}(n-k,k)}{\mathbf q}^{{\rm Des}(T)}\right).\end{equation}

We conclude, using Gessel's theorem (Theorem \ref{gessel} above), that
\begin{equation}
    \mathcal{Q}_n(\mathcal{PS}et(n,b))=\sum\limits_{k=0}^{\left\lfloor\frac{n}{2}\right\rfloor} e^{(b)}_{n-k,k} s_{(n-k,k)}
    \end{equation}
\end{thm}

Note that some of the coefficients $S_{\geq 2}(n-t,b-t)$ might be $0$ when $b<t$, as can be seen in Example \ref{example blocks procedure k geq 2} below.

\begin{proof}
We apply the process described in Sections \ref{section: preimage desing} and \ref{section case k=1} above for producing the pre-image of a given tableau $T\in {\rm SYT}(n-k,k)$ with $k>0$, forcing the restriction that the number of blocks of each element of ${\varphi}^{-1}(\{T\})$ is exactly $b$. This causes a two-sided limitation on the number $t$ of possible singletons in each $\pi \in {\varphi}^{-1}(\{T\})$:
\begin{equation}\label{limit on num of singletons}
\max\{{\rm des}(T), 2b-n\} \leq t \leq b-1,
\end{equation}
due to the following reasons:
\begin{itemize}
\item The right inequality is by definition of ${\rm Desing}(\pi)$, as at least one non-singleton block must exist.

\medskip

\item The left inequality is due to two requirements: first, by definition, the number of singletons should be no less than the number of descents, and second, the number of elements in non-singleton blocks $n-t$ should be at least twice the number of remaining blocks $b-t$, as each such block contains at least two elements.
\end{itemize}

From here on, the proof goes almost verbatim as the steps of the proof of Theorem \ref{main theorem} with the following two small technical amendments, forced by the fixed number of blocks $b$:
\begin{enumerate}
\item Replace each appearance of $\mathcal{PS}et(n)$ by $\mathcal{PS}et(n,b)$. Note that this might cause some of the sets $L_{n-k-i}$ to vanish, as can be seen in Example \ref{example blocks procedure k geq 2} below.
\item Replace each appearance of ${\rm Bell}_{\geq2}(n-t)$ by $S_{\geq 2}(n-t,b-t)$, which implies  the replacement of $e_{n-k,k}$ by $e^{(b)}_{n-k,k}$.
\end{enumerate}

\vspace{-25pt}\end{proof}

\begin{remark}
Note that $e_{n,0}^{(b)}$ counts the number of set partitions $\pi$ of $[n]$ having $b$ blocks and satisfying ${\rm desing}(\pi)=0$. It can be easily seen
from the free coefficient in Equation (\ref{eqn main theorem blocks}).
\end{remark}

\begin{example}\label{example blocks procedure k geq 2}
Fix the number of blocks to be $b=4$ and let $$T={\begin{ytableau} 1&3&4&7&8 \\ 2&5&6 \end{ytableau}} \ .$$

In Example \ref{example procedure k geq 2} above, we had $R_1=\{1,3,4,7,8\}$ and $R_2=\{2,5,6\}$, and so we obtained the single set partition:
$\pi_1=1 \mid 256 \mid 3 \mid 4 \mid 7 \mid 8$. However, in our case, there are $S_{\geq 2}(8-5,4-5)=S_{\geq 2}(3,-1)=0$ such partitions, as there are already $5$ singletons, so it is impossible to have a set partition with $4$ blocks.

As before, we have $H(\pi)=\{6\}$, so that ${\rm Rem}(T)=\{7,8\}$.

We first move $7$ from $R_1$ to $R_2$ to get $R_1^{(1)}=\{1,3,4,8\}$ and $R_2^{(1)}=\{2,5,6,7\}$. Also in this case, there are $S_{\geq 2}(8-4,4-4)=S_{\geq 2}(4,0)=0$ such set partitions, as there are still $4$ singletons.

Next, we move $8$ from $R_1^{(1)}$ to $R_2^{(1)}$ to get $R_1^{(2)}=\{1,3,4\}$ and $R_2^{(2)}=\{2,5,6,7,8\}$. In this case, there is  $S_{\geq 2}(8-3,4-3)=S_{\geq 2}(5,1)=1$ such set partition:
$$\pi'_1=1 \mid 25678 \mid 3 \mid 4.$$


Hence, the pre-image of $T$ is the single set partition $\{\pi'_1\}$: $$e_{5,3}^{(4)}=S_{\geq 2}(3,-1)+S_{\geq 2}(4,0)+S_{\geq 2}(5,1)=0+0+1=1.$$
\end{example}

\subsection{Combinatorial perspectives and associated sequences with some open questions}
In Tables \ref{tab:s_nk_blocks} and \ref{tab:s_nk_blocks_1}, we present some numerical values of the coefficients $e^{(b)}_{n-k,k}$.

\begin{table}[H]
{\tiny
$\begin{NiceArray}{|w{c}{0.4cm}|c;cc;ccc;cccc;ccccc|}
\hline
\Block{2-1}{\diagbox{\,\,\,\,n}{b\,}} & 1& 2 && &3 &&& 4 &&& &&5&&   \\
& k=0 & k=0 & k=1 &k=0&k=1&k=2 &k=0&k=1&k=2&k=3&k=0&k=1&k=2&k=3&k=4\\
\hline
2  & \fbox{\bf 1} & 1 &&&&&&&&&&&&& \\
3  & 1 & 1&1 & 1 &&&&&&&&&&& \\
4  & 1 &  \fbox{\bf 4}&1  & 1&1&1 & 1 &&&&&&&& \\
5  & 1 &  11& 1 & 4&4&1 & 1&1&1 && 1 &&&&\\
6  & 1 &  26&1 & \fbox{\bf 26}& 11&1 & 4&4&4&1  & 1&1&1&& \\
7  & 1 &  57& 1 & 131&26&1 & 26&26&11&1 & 4&4&4&4& \\
8  & 1 &  120&1  & 547&57&1 & \fbox{\bf 236}&131&26&1  & 26&26&26&11&1  \\
9  & 1 &  247&1 & 2038&120&1 & 1807&547&57&1 & 236&236&131&26&1 \\
10 & 1 & 502& 1 & 7072&247&1 & 11488&2038&120&1 & \fbox{\bf 2752}&1807&547&57&1 \\
11 & 1 & 1013&1 & 23437&502&1 & 64052&7072&247&1 & 28813&11488&2038&120&1 \\
\hline
\end{NiceArray}$}
\caption{Numerical values of $e^{(b)}_{n-k,k}$ for $2\leq b \leq 5$ and small values of $n,k$.
Note that $k<b$.}
\label{tab:s_nk_blocks}
\end{table}

\begin{table}[H]
{\tiny$\hspace{-40pt}
\begin{NiceArray}{|w{c}{0.4cm}|cccccc;ccccccc;l;l|}
\hline
\Block{2-1}{\diagbox{\,\,\,\,n}{b\,}} &&&6&&&&&&& 7 &&&  & \qquad \qquad 8 & \qquad 9 \\
& k=0 & k=1 & k=2 &k=3&k=4&k=5 &k=0&k=1&k=2&k=3&k=4&k=5&k=6&&\\
\hline
6 & 1 &&&&&&&&&&&& &&\\
7 &1 &1&1 &&&& 1 &&&&&&&&\\
8 &   4&4&4&4&3&&  1&1&1&&&& & 1 & \\
9  & 26&26&26&26&11& & 4&4&4&4&3&& & 1,1,1 & 1\\
10 &  236&236&236&131&26&1 & 26&26&26&26&26&10 && 4,4,4,4,3 & 1,1,1\\
11 & 2752&2752& 1807&547&57&1 & 236& 236&236&236&131&26 && 26,26,26,26,26,25 & 4,4,4,4,3\\\hline
\end{NiceArray}$}

\caption{Numerical values of $e^{(b)}_{n-k,k}$ for $6 \leq b \leq 9$ and small values of $n,k$ (for $n<6$ the values are zero).
}
\label{tab:s_nk_blocks_1}
\end{table}

As the coefficients $e_{n-k,k}^{(b)}$ are refinements of $e_{n-k,k}$, we have:
$e_{n-k,k}=\sum\limits_{b=1}^n e_{n-k,k}^{(b)}$. 

The following claim explains the phenomenon that the second sub-column ($k=1$) of the column $b=3$ is equal to the first sub-column ($k=0$) of the column $b=2$, shifted one place down, and so on:
\begin{claim}\label{claim_columns_enk}
If $k\leq n-b$ or $k> 2(n-b)$, 
then: $e_{n-k,k}^{(b)} = e_{n-k,k+1}^{(b+1)}$.
\end{claim}

\begin{proof}
Compute by Definition \ref{def: e_{n-k,k}^{(b)}}:
\begin{eqnarray*}
e_{n-k,k}^{(b)}&=&\sum\limits_{t=k}^{n-k}S_{\geq 2}(n-t,b-t)=\\
&=&\left[\sum\limits_{t=k}^{n-k-1}S_{\geq 2}(n-t,b-t)\right]+S_{\geq 2}(k,b-(n-k))=\\
&=&
\left[\sum\limits_{t=k+1}^{n-k}S_{\geq 2}(n-(t-1),b-(t-1))\right]+S_{\geq 2}(k,b-(n-k))=
\\
&=&e_{n-k,k+1}^{(b+1)}+S_{\geq 2}(k,b-(n-k)).
\end{eqnarray*}

Hence, we have that $e_{n-k,k}^{(b)}=e_{n-k,k+
1}^{(b+1)}$ if and only if $S_{\geq 2}(k,b-(n-k))=0$, which happens if and only if either $2(b-(n-k))>k$ or $b\leq n-k$. In terms of $k$, the last two restrictions are: $k>2(n-b)$ or $k\leq n-b$, as formulated. (It seems that the condition $k>2(n-b)$ is never fulfilled, but we have no proof.)
\end{proof}

\begin{question}\label{question equality columns enk}
Using the combinatorial meaning of $k$, given in Corollary \ref{cor: combinatorial meaning k} above, can one find a bijection between the set partitions of $[n]$ counted by $e_{n-k,k}^{(b)}$ and the set partitions of $[n]$ counted by $e_{n-k,k+1}^{(b+1)}$?
\end{question}

\subsubsection{The values $e_{n,0}^{(b)}$}

By Claim \ref{claim_columns_enk} above, we have that the first sub-column ($k=0$) for each value of $b$ almost determines all the other sub-columns (for $k>0$). Therefore, let us concentrate only on the sub-column $k=0$ for each $b$, as presented in Table \ref{tab:s_nk0_blocks}.

\begin{table}[H]
{\small
$\begin{NiceArray}{|w{c}{0.4cm}|c;c;c;c;c;c;c;c;c;c;c|}
\hline
\Gape[0.35cm][0.35cm]{\diagbox{\,\,n}{b\,}} &  1 & 2 & 3 & 4&5&6 & 7 &8 &9 &10 &11 \\
\hline
2  & \fbox{\bf 1} & 1 &&&&&&&&& \\
3  & 1 &  1 & 1 &&&&&&&& \\
4  & 1 &  \fbox{\bf 4} & 1  & 1 &&&&&&& \\
5  & 1 &  11& 4&1 &1 &&&&&&\\
6  & 1 &  26& \fbox{\bf 26}& 4& 1&1 &&&&& \\
7  & 1 & 57 & 131 & 26& 4&1&1&&&& \\
8  & 1 & 120 & 547 & \fbox{\bf 236} & 26 & 4 &1&1&&&  \\
9  & 1 & 247& 2038& 1807 & 236& 26 & 4 & 1 &1 && \\
10 & 1 & 502& 7072 & 11488 &  \fbox{\bf 2752}& 236& 26 & 4 & 1 &1& \\
11 & 1 & 1013 & 23437 &  64052 &  28813 &  2752& 236& 26 & 4 & 1 &1\\
\hline
\end{NiceArray}$}

\caption{Numerical values of $e^{(b)}_{n,0}$ for small values of $n$ and $b$.
}
\label{tab:s_nk0_blocks}
\end{table}

The sequence of elements in Table \ref{tab:s_nk0_blocks} which are squared is the sequence\break $\left\{e^{(b)}_{2b,0}\right\}=\{1,4,26,236,2752,\dots\}$ (sequence A000311 in OEIS \cite{OEIS}), which is known to count the {\it total partitions} on $n$ (see Stanley \cite[Example 5.2.5]{EC2}).
An exercise in Comtet \cite[p. 224]{Comtet}
(in the context of counting {\it Schr\"oder systems}) showed that the general element of this sequence is indeed the sum of associated Stirling numbers: $e_{2b,0}^{(b)}=\sum\limits_{t=0}^{b-1} S_{\geq 2} (2b-t,b-t)$ as defined above; see also \cite[Example 7.3]{Schreiber}.

\medskip

Note that in Table \ref{tab:s_nk0_blocks}, for each $n$, every diagonal $\left\{e_{n+i,0}^{(1+i)}\right\}_{i\in \mathbb{N}_0}$ becomes constant after passing its corresponding squared element $e_{2b,0}^{(b)}$.
The next claim proves this phenomenon:

\begin{claim}\label{claim en0b}
For $n \leq 2b$, we have that:
$e_{n+1,0}^{(b+1)}=e_{n,0}^{(b)}$.
\end{claim}

\begin{proof}
Recall that $e_{n,0}^{(b)}$ counts the set partitions $\pi$ of $[n]$ having $b$ blocks with\break ${\rm desing}(\pi)=0$, which means the singletons appear only at the end of $\pi$. The mapping that sends $\pi$ (counted by  $e_{n,0}^{(b)}$) to $\pi \cup \{\{n+1\}\}$ (counted by $e_{n+1,0}^{(b+1)}$) is a bijection when $n\leq 2b$, since each partition counted by  $e_{n+1,0}^{(b+1)}$ (for $n\leq 2b$) must contain the singleton $\{n+1\}$.
\end{proof}

Note that this also explains the phenomenon appearing in Tables \ref{tab:s_nk_blocks} and \ref{tab:s_nk_blocks_1}, that numbers from the sequence $\{1,4,26,236,2752,\dots\}$ (sequence A000311 in OEIS \cite{OEIS}) appear more frequently as values of $e_{n-k,k}^{(b)}$.

\subsubsection{The sequence $\left\{e_{n,0}^{(2)}\right\}$ and its connection to the sequence of  Eulerian numbers}

The second column of Table \ref{tab:s_nk0_blocks} forms the sequence $$\left\{e^{(2)}_{n,0}\right\}=\{1,1,4,11,26,57,120,247, \dots\},$$ whose $n$'th element counts the number of set partitions $\pi$ of $[n]$ having two blocks with ${\rm desing}(\pi)=0$. It is easy to observe that
\begin{equation}\label{Eqn: en02}
e_{n,0}^{(2)}=S_{\geq 2}(n,2)+1, \end{equation}
as the condition ${\rm desing}(\pi)=0$ with the requirement that $\pi$ has two blocks allows only the existence of the singleton $\{n\}$ (see the definition of $S_{\geq 2}(n,2)$ and the second column of Table \ref{table A008299}).
Note also that $S_{\geq 2}(n,2)=2^{n-1}-1-n$, as we have to choose the elements of one of the two non-empty blocks, for which we have $2^{n-1}-1$ possibilities by symmetry, and then we have to throw out the $n$ cases where one of the blocks is a singleton.
Therefore, we have: $e_{n,0}^{(2)}=2^{n-1}-n$.

On the other hand, note that the number $2^{n-1}-n$ is also the celebrated Eulerian number $A(n-1,1)$, which counts permutations of $n-1$ elements having exactly one descent in their one-line notation (see sequence A000295 in OEIS \cite{OEIS}).
Indeed, a permutation of $[n-1]$ having exactly one descent can be formed by choosing a non-empty subset $B \subsetneqq [n-1]$, which is not of the form $\{1,2,\dots,i\}$ for all $1\leq i\leq n-1$, and concatenating to its elements, the elements of its complement in $[n-1]$, where both subsets are ordered increasingly (see $\sigma_i$ in Example \ref{example-eulerian} below).
We have in summary that $$e_{n,0}^{(2)}=A(n-1,1).$$

\medskip

This equality deserves a combinatorial proof, which we supply below,  based on a simple direct bijection between the following two sets, counted by $e_{n,0}^{(2)}$ and $A(n-1,1)$, respectively:
\begin{itemize}
\item Let $\mathcal{A}(n-1,1)$ be the set of all permutations on $n-1$ elements having one descent in their one-line notation, i.e. each element of  $\mathcal{A}(n-1,1)$  is of the form\break $\sigma=a_1\cdots a_{n-1}$, where for some $1\leq i\leq n-2$,
$A_1=\{a_1,\dots ,a_i\}$ and\break $A_2=\{a_{i+1},\dots ,a_{n-1}\}$
are the two runs (consecutive increasing sequences) of $\sigma$.

\medskip

\item Let $\mathcal{B}_n$ be the partitions in ${\mathcal{PS}}et(n,2)$ having ${\rm desing(\pi)}=0$.
\end{itemize}

\begin{claim}[A combinatorial proof for $e_{n,0}^{(2)}=A(n-1,1)$]\label{combinatorial claim for A(n-1,1)}\ \\
Define the following mapping $\varphi:\mathcal{A}(n-1,1) \rightarrow \mathcal{B}_n$ by:
$$\varphi(\sigma) = \left\{ \begin{array}{ll}
\{\{1,\dots, i\},\ \{i+1, \dots, n\}\}& \ \ A_1=\{i\} \mbox{ for some } 2\leq i \leq n-1\medskip\\
\{A_1,\ A_2\cup\{n\}\}& \ \ \mbox{otherwise.} \\
\end{array} \right.$$

Then, $\varphi$ is a bijection.    \end{claim}

\begin{proof}
Note that $\varphi(\sigma) \in \mathcal{B}_n$, as the only singleton which can appear is $\{n\}$, implying ${\rm desing}(\varphi(\sigma))=0$. Indeed:
\begin{itemize}
\item The first part of $\varphi(\sigma)$ is never a singleton as $i\geq 2$ in the first case and $A_1$ is not a singleton in the second case.

\medskip

\item The second part of $\varphi(\sigma)$ is not a singleton in the second case as $A_2\neq\emptyset$, and it can be a singleton in the first case only when $i=n-1$, in which case we get the singleton $\{n\}$.
\end{itemize}

Now, it is easy to recover $\sigma \in \mathcal{A}(n-1,1)$, such that  $\varphi(\sigma)=\pi =\{B_1,B_2\} \in \mathcal{B}_n$ (assuming $n \in B_2$) as follows:
$$\varphi^{-1}(\{\{B_1,B_2\}\}) = \left\{ \begin{array}{ll}
i 1 2 \cdots (i-1)(i+1) \cdots (n-1) & \ \ B_1=\{1,\dots, i\} \mbox{ for some } 2\leq i \leq n-1\medskip\\
C_1C_2& \ \ \mbox{otherwise,} \\
\end{array} \right.$$
where $C_i$ are the elements of the set $B_i$ ordered increasingly.
\end{proof}

\begin{example}\label{example-eulerian}
Let $\sigma_1=3471256\in A(7,1)$. Then we have: $$A_1=\{3,4,7\}, A_2=\{1,2,5,6\},$$
so that $\varphi(\sigma_1)=347 \mid 12568 \in \mathcal{B}_8$, by the second case.

As another example, if $\sigma_2=7123456\in \mathcal{A}(7,1)$, then we have:
$$A_1=\{7\}, A_2=\{1,2,3,4,5,6\},$$
so that: $\varphi(\sigma_2)=1234567 \mid 8 \in \mathcal{B}_8$, by the first case.
\end{example}

\subsubsection{The connection between the values $\left\{e_{n,0}^{(b)}\right\}$ and the sequence A124324}

The next result connects between the values $\left\{e_{n,0}^{(b)}\right\}$ (see Table \ref{tab:s_nk0_blocks}) and the values of sequence A124324 in OEIS \cite{OEIS},  which count the number of set partitions of $[n]$ having $m$ blocks of size greater than $1$ (and possibly some additional singletons). The first values of sequence
A124324 are presented in table form in Table \ref{tab:A124324}, and we denote the value in the entry $(n,m)$ by $T(n,m)$.

\begin{table}[H]
{\small
$\begin{NiceArray}{|w{c}{0.4cm}|c;c;c;c;c;c|}
\hline
\Gape[0.35cm][0.35cm]{\diagbox{\,\,n}{m\,}} &  0 &1 & 2 & 3 & 4&5 \\
\hline
2 &  1 &    1 & &&& \\
3 &  1 &    4 & &&& \\
4 &  1 &   11 &     3 &&& \\
5 &  1 &   26 &    25 &&& \\
6 &  1 &   57 &   130 &    15 && \\
7 &  1 &  120 &   546 &   210 && \\
8 &  1 &  247 &  2037 &  1750 &   105 & \\
9 &  1 &  502 &  7071 & 11368 &  2205 & \\
10 & 1 & 1013 & 23436 & 63805 & 26775 & 945\\
\hline
\end{NiceArray}$}
\caption{Table form of sequence A124324 in OEIS; its entries are denoted by $T(n,m)$.}
\label{tab:A124324}
\end{table}

The connection between the values $\left\{e_{n,0}^{(b)}\right\}$ and the values $T(n,m)$ is given in the following proposition:

\begin{prop}\label{prop e to T}
For each $n$ and $b$, we have:
$$e_{n,0}^{(b)}=\left\{
\begin{array}{ll}
\sum\limits_{i=0}^{\left\lfloor\frac{b-1}{2}\right\rfloor} T(n-1-2i,b-1-2i) & \ \ n \neq 2b \medskip \\
\sum\limits_{i=0}^{\left\lfloor\frac{b}{2}\right\rfloor} T(n-2i,b-2i) & \ \  n=2b \\
\end{array}\right. .$$
\end{prop}

\medskip

For proving this proposition, we need the following combinatorial lemma.

\begin{lemma}\label{lemma T to S}
Let $b>1$. Then:
\begin{enumerate}
\item For $n-1\neq 2(b-1)$, we have:
$$T(n-1,b-1)=S_{\geq2}(n,b)+S_{\geq2}(n-1,b-1).$$

\medskip

\item $T(2b,b)=S_{\geq2}(2b,b)$.
\end{enumerate}
\end{lemma}

\begin{proof}
(1) Let $\pi$ be a set partition, counted by $T(n-1,b-1)$.
If $\pi$ has no singletons at all, then $\pi$ is counted by $S_{\geq2}(n-1,b-1)$.
Otherwise, the singletons of $\pi$ are $\{a_1\},\dots,\{a_p\}$ for some $p\geq 1$, and therefore $\pi$ has $b-1+p$ blocks. Then, $\pi$ is mapped bijectively to the set partition $\pi'$ of $[n]$ defined by:
$$\pi' = \left(\pi -\{\{a_1\},\dots,\{a_p\}\}\right)\cup \{\{a_1,\dots,a_p,n\}\},$$
which is counted by $S_{\geq2}(n,b)$.

(2) This equality is obvious, as there are no singletons.
\end{proof}

\medskip

\begin{example}
Let us illustrate the proof of Lemma \ref{lemma T to S}(1) for $n=7$ and $b=4$, i.e. we explain why:
$$130 = T(6,2)= S_{\geq2}(7,3)+S_{\geq2}(6,2)= 105+25.$$
A partition $\pi$ counted by $T(6,2)$ can have one of the following types of blocks: $$(3,3),(4,2), (2,2,1,1)  \mbox{ and } (3,2,1).$$
The set partitions of the first two types are obviously counted by $S_{\geq2}(6,2)$.
The set partitions of the other two types are in bijection with the set partitions counted by $S_{\geq2}(7,3)$, and we provide here two examples, where we unite all the singletons with $7$:
\begin{align*}
125 \mid 3 \mid 46 \  &\mapsto \ 125 \mid 37 \mid 46,\\
12 \mid 3 \mid 46 \mid 5 \  &\mapsto \ 12 \mid 357 \mid 46.
\end{align*}
\end{example}

\begin{proof}[Proof of Proposition \ref{prop e to T}]
We start with the case $n=2b$.  Note that\break $S_{\geq 2}(2b,b)=T(2b,b)$ by Lemma \ref{lemma T to S}(2).
In the case of odd $b$, we can write by Lemma \ref{lemma T to S}:
\begin{eqnarray*}
e_{2b,0}^{(b)}& = &\sum\limits_{t=0}^{n}S_{\geq 2}(2b-t,b-t) = \sum\limits_{t=0}^{b-1}S_{\geq 2}(2b-t,b-t) = \\
& = & S_{\geq 2}(2b,b) +\left[S_{\geq 2}(2b-1,b-1)+S_{\geq 2}(2b-2,b-2)\right] + \cdots +\\
&  &  \qquad + [S_{\geq2} (b+2,2)+S_{\geq2} (b+1,1)]=\\
& = & T(2b,b)+T(2b-2,b-2) + \cdots + T(b+1,1)=\sum\limits_{i=0}^{\left\lfloor\frac{b}{2}\right\rfloor} T(n-2i,b-2i).
\end{eqnarray*}

In the case of even $b$, we can write again by Lemma \ref{lemma T to S}:
\begin{eqnarray*}
e_{2b,0}^{(b)}& = &\sum\limits_{t=0}^{n}S_{\geq 2}(2b-t,b-t) = \sum\limits_{t=0}^{b-1}S_{\geq 2}(2b-t,b-t) = \\
& = & S_{\geq 2}(2b,b) +\left[S_{\geq 2}(2b-1,b-1)+S_{\geq 2}(2b-2,b-2)\right] + \cdots +\\
&  &  \qquad + [S_{\geq2} (b+3,3)+S_{\geq2} (b+2,2)]+S_{\geq2} (b+1,1)=\\
& = & T(2b,b)+T(2b-2,b-2) + \cdots + T(b+2,2)+T(b,0)=\\
& = & \sum\limits_{i=0}^{\left\lfloor\frac{b}{2}\right\rfloor} T(n-2i,b-2i),
\end{eqnarray*}
since $S_{\geq2} (b+1,1)=T(b,0)=1$.

\medskip

For the case that $n\neq 2b$, we similarly write for even $b$ by Lemma \ref{lemma T to S}(1):

\begin{eqnarray*}
e_{n,0}^{(b)}& = &\sum\limits_{t=0}^{n}S_{\geq 2}(n-t,b-t) = \\
& = & [S_{\geq 2}(n,b) +S_{\geq 2}(n-1,b-1)]+[S_{\geq 2}(n-2,b-2)+S_{\geq 2}(n-3,b-3)] + \cdots +\\
&  &  \qquad + [S_{\geq2} (n-b+2,2)+S_{\geq2} (n-b+1,1)]=\\
& = & T(n-1,b-1)+T(n-3,b-3) + \cdots + T(n-b+1,1)=\\
&=&\sum\limits_{i=0}^{\left\lfloor\frac{b}{2}\right\rfloor} T(n-1-2i,b-1-2i).
\end{eqnarray*}

In the case of odd $b$, we have:
\begin{eqnarray*}
e_{n,0}^{(b)}& = &\sum\limits_{t=0}^{n}S_{\geq 2}(n-t,b-t)=  \\
& = &  [S_{\geq 2}(n,b) +S_{\geq 2}(n-1,b-1)]+[S_{\geq 2}(n-2,b-2)+S_{\geq 2}(n-3,b-3)] + \cdots +\\
&  &  \qquad + [S_{\geq2} (n-b+3,3)+S_{\geq2} (n-b+2,2)]+S_{\geq2} (n-b+1,1)=\\
& = & T(n-1,b-1)+T(n-3,b-3) + \cdots + T(n-b+2,2)+T(n-b,0)=\\
& = & \sum\limits_{i=0}^{\left\lfloor\frac{b}{2}\right\rfloor} T(n-1-2i,b-1-2i),
\end{eqnarray*}
since $S_{\geq2} (n-b+1,1)=T(n-b,0)=1$.
\end{proof}

\medskip

\begin{remark}
The sequence $T(n,m)$ mentioned above has another interesting combinatorial interpretation as the number of run-sorted permutations over $[n]$ having $m$ runs (see \cite{BeMa}), where $\pi\in \mathcal{S}_n$ is called {\em run-sorted}
if there exists a set partition $\tau=B_1 \mid \cdots \mid B_t$ of $[n]$ with ${\rm max}(B_i)>{\rm min}(B_{i+1})$ for all $1\leq i \leq t-1$,  such that $\pi$ is obtained by $\tau$ by the process of flattening (i.e. removing the
bars between the blocks; such set partitions are called there {\em merging-free}).
Note that the condition $\max(B_i)>\min(B_{i+1})$ in a set partition $\tau$ implies ${\rm desing}(\tau)=0$, but the converse is not true.
\end{remark}

\begin{remark}
The sequence A112493 in OEIS \cite{OEIS} locates the values of Table \ref{tab:A124324} in different positions, see Table \ref{tab:A112493} (by a comment of J\"orgen Backelin in the OEIS there).
The connection between the entries in the two tables is:
$$A124324(n+m,m)=A112493(n,m).$$

\begin{table}[H]
{\small
$\begin{NiceArray}{|w{c}{0.4cm}|c;c;c;c;c;c;c;c;c|}
\hline
\Gape[0.35cm][0.35cm]{\diagbox{\,\,n}{m\,}} &  0 &1 & 2 & 3 & 4&5&6&7&8 \\
\hline
1 &   1  &  1  &&&&&&& \\
2 &   1  &  4  &  3 &&&&&& \\
3 &   1  &  11 & 25 & 15 &&&&& \\
4 &   1  & 26  & 130 & 210 & 105 &&&& \\
5  &   1  & 57 & 546 & 1750 & 2205 & 945 &&& \\
6  & 1 & 120 & 2037 & 11368 &    26775  &  27720 &  10395  && \\
7  &   1  &  247 & 7071 & 63805 &   247555   & 460845  &   405405  &   135135 & \\
8  &  1  &   502  &  23436 &  325930  & 1939630  & 5735730  &  8828820 &  6756750 & 2027025 \\
\hline
\end{NiceArray}$}
\caption{Table form of sequence A112493 in OEIS; its entries are denoted by $R(n,m)$.}
\label{tab:A112493}
\end{table}

Following Proposition \ref{prop e to T} and the above-mentioned connection between Tables \ref{tab:A124324} and \ref{tab:A112493}, we actually have that the $b$'th column of Table \ref{tab:s_nk0_blocks} (except for the entry $(2b,b)$)
is the sum of columns $1, 3, \dots, b$ in Table \ref{tab:A112493} if $b$ is odd, and  the sum of columns $0,2, 4, \dots, b$ in Table \ref{tab:A112493} if $b$ is even.
\end{remark}

\subsubsection{A connection to second-order Eulerian numbers}

The sequence A112493 (appearing in Table \ref{tab:A112493}) is defined for $n \geq 0$ and $0 \leq k \leq n$ by: $$R(n, k) = \sum\limits_{j=0}^k {{n-j} \choose  {n-k}}E_2(n, j),$$ where $E_2(n,j)$ are the second-order Eulerian numbers; see sequences A201637 and A340556 in OEIS \cite{OEIS} (for a survey, see Deza's book \cite[Chapter 4]{DezaBook}).
Hence, it follows that
\begin{equation}\label{eqn: Tnk by E2}
T(n,k)=R(n-k,k)=\sum\limits_{j=0}^{n-k} {{n-k-j} \choose  {n-2k}}E_2(n-k, j),
\end{equation}
where $T(n,k)$ are the values taken from Table \ref{tab:A124324}.

\begin{remark}\label{remark6.16}
In Lemma \ref{lemma T to S}, we established a connection between the values $T(n,k)$ of the sequence A124324 (Table \ref{tab:A124324} above) and the associated Stirling numbers $S_{\geq 2}(n,k)$  (Table \ref{table A008299} above). By substituting Equation (\ref{eqn: Tnk by E2}) in the equations of Lemma \ref{lemma T to S}, we get the following equalities:
\begin{enumerate}
\item For $n-1\neq 2(b-1)$, we have:
\begin{equation}\label{Eqn:Question (a)}
[T(n-1,b-1)=\ ]\sum\limits_{j=0}^{n-b} {{n-b-j} \choose{n-2b+1}}E_2(n-b, j)=S_{\geq2}(n,b)+S_{\geq2}(n-1,b-1),
\end{equation}
\medskip

\item
\begin{equation}\label{Eqn:Question (b)}
[T(2b,b)=\ ]\sum\limits_{j=0}^{b} E_2(b, j)=S_{\geq2}(2b,b)\ [ =(2b-1)(2b-3)\cdots 1].
\end{equation}
\end{enumerate}

Note that Equation (\ref{Eqn:Question (b)}) has the following nice combinatorial meaning: The left side counts all the {\em Stirling permutations of order $b$}, which are permutations of the multi-set $\{1,1,2,2,\dots,b,b\}$ satisfying that for each $i$, every entry between the two copies of $i$ is larger than $i$ (introduced by Gessel-Stanley \cite{GeSt1978}). The right hand side is the number of set partitions of the set $[2b]$ into $b$ blocks of size at least $2$ and hence into $b$ blocks of size {\em exactly} $2$, i.e., perfect matchings on $[2b]$. Both sets have cardinality $(2b-1)(2b-3)\cdots 1$. Connections between Stirling permutations and perfect matchings have also been studied by Ma and Yeh \cite{MaYeh}.

The following recursive insertion algorithm is an  analogue of the classical gap-insertion construction for Stirling permutations.
A Stirling permutation of order $b$ is obtained recursively by inserting the pair $bb$ into one of the $2b-1$ gaps of a Stirling permutation of order $b-1$.
The algorithm illustrates how to establish a bijection between perfect matchings and Stirling permutations.
Given a perfect matching $M$ on $[2b]$, if $\{i,2b\}\in M$, then the pair $bb$ will be inserted in the $i$'th gap. Now we delete the block $\{i,2b\}$ from $M$ and standardize. Then we proceed recursively on the resulting perfect matching on $[2b-2]$.

We illustrate this algorithm by a small example: Consider the perfect matching\break
$
M=\{\{1,3\},\{2,6\},\{4,7\},\{5,8\}\}
$ of $[8]$:
\begin{itemize}
\item We remove the block containing \(8\), namely \(\{5,8\}\). This records
that, at the final step, the pair \(44\) should be inserted into the
$5$'th gap. Standardizing the remaining
blocks, we obtain the perfect matching
\(
\{\{1,3\},\{2,5\},\{4,6\}\}.
\)
\item Next, remove \(\{4,6\}\). This records
that \(33\) should be inserted into the $4$'th gap. Standardizing the remaining
blocks, we obtain the matching
\(
\{\{1,3\},\{2,4\}\}.
\)
\item Next, remove \(\{2,4\}\). This records
that \(22\) should be inserted into the $2$'nd gap.
\item Finally, standardizing the remaining
block, we obtain
\(
\{\{1,2\}\},
\)
which corresponds to the Stirling permutation
\(
11.
\)

\end{itemize}

We now rebuild the Stirling permutation by inserting adjacent equal pairs
into the recorded gaps. Starting with
\(
11,
\)
we insert \(22\) into the $2$'nd gap, then we insert \(33\) into the $4$'th gap, and finally, we insert \(44\) into the $5$'th gap:
\[
11 \longmapsto 1\ \fbox{$\mathbf{22}$}\ 1 \longmapsto 122\ \fbox{$\mathbf{33}$}\ 1 \longmapsto 1223\ \fbox{$\mathbf{44}$}\ 31.
\]
Hence, by the recursive algorithm, we have:
\[
\{\{1,3\},\{2,6\},\{4,7\},\{5,8\}\}
\longleftrightarrow
12234431.
\]
\end{remark}

\medskip

The following questions are now natural:
\begin{question}\label{question comb connection 2nd eulerian to stirling}
Similar to the combinatorial proof of Claim \ref{combinatorial claim for A(n-1,1)} and the combinatorial interpretation appearing the second part of Remark \ref{remark6.16}, can one find a combinatorial explanation for Equation (\ref{Eqn:Question (a)}) above?
\end{question}

\begin{question}\label{question comb connection 2nd eulerian to en0b}
Substituting the values $T(n,k)$ from Equation (\ref{eqn: Tnk by E2}) in the formulation of Proposition \ref{prop e to T},
we have for all $n$ and $b$:
$$e_{n,0}^{(b)}=\left\{
\begin{array}{ll}
\sum\limits_{i=0}^{\left\lfloor\frac{b-1}{2}\right\rfloor} \sum\limits_{j=0}^{n-b} {{n-b-j} \choose  {n-2b+2i+1}}E_2(n-b, j) & \ \ n \neq 2b \medskip \\
\sum\limits_{i=0}^{\left\lfloor\frac{b}{2}\right\rfloor} \sum\limits_{j=0}^{b} {{b-j} \choose  {2i}}E_2(b, j) & \ \  n=2b \\
\end{array}\right. .$$
Can one find a combinatorial explanation for this equality?
\end{question}

\section{The distribution of the Desing parameter}\label{section: distribution Desing parameter}

In this section, we give another aspect of the Desing parameter, namely, we count the number of set partitions having a given Desing set, see Definition \ref{first definitions for Desing}.

\medskip

Given a set partition $\pi$ of $[n]$, we encode its pattern of singletons by a binary string: $x=x_1x_2\cdots x_n\in\{0,1\}^n$, where
$x_i=1$ if $i \in {\rm Sing}(\pi)$ and $x_i=0$ otherwise.
Then a given Desing set
$D=\{d_1<d_2<\cdots<d_s\}\subseteq\{1,\dots,n-1\}$
forces $(x_{d_j},x_{d_j+1})=(1,0)$ for each $d_j\in D$, and forbids any additional occurrence of the pattern $10$.
Note that between consecutive forced $10$'s, the binary string $x$ has to be of the form $0^\ast1^\ast$ (zeros followed by ones); i.e., extra $1$'s can only appear as suffixes of the gaps.

Let $t=\sum\limits_{i=1}^n x_i$ be the total number of singletons in $\pi$. For a given set $D$ as above, let $F(D,t)$ be the number of binary strings $x\in\{0,1\}^n$ with exactly $t$ $1$'s, whose only occurrences of $10$ appear at indices in $D$.
Write the \emph{gap-lengths}
$$
L_0=d_1-1,\qquad L_j=d_{j+1}-d_j-2\ (1\le j\le s-1),\qquad L_s=n-d_s-1.
$$
Inside the $j$'th gap, we may append a suffix of exactly $b_j$ $1$'s, where $0\le b_j\le L_j$. As the $1$'s  must be suffixes, there is a unique placement for each $b_j$. Hence, $F(D,t)$ is the number of bounded weak compositions
$$b_0+b_1+\cdots+b_s \;=\; t-s,\qquad 0\le b_j\le L_j.$$
So we have:
\begin{lemma}
\label{lem:FGEN}
Let $D=\{d_1<\cdots<d_s\}$ with gap-lengths $L_0,\dots,L_s$ as above. Then:
$$\sum_{t} F(D,t)\,x^{\,t-s} \;=\; \prod_{j=0}^{s} \bigl(1+x+x^2+\cdots+x^{L_j}\bigr).$$
\end{lemma}

Note that $F(D,t)>0$ if and only if $s\le t\le n-s$.

After fixing the $t$ singletons, there remain $m=n-t$ elements to be partitioned into blocks of size at least $2$.

For each set $D\subseteq [n-1]$, we are now able to count the number of set partitions $\pi$ of $[n]$ having $b$ blocks and having ${\rm Desing}(\pi)=D$:
\begin{prop}
\label{thm:master}
Let $n$ and $b$ be positive numbers and let $D\subseteq\{1,\dots,n-1\}$ have no consecutive integers. Then the number of set partitions $\pi$ of $[n]$ into $b$ blocks with ${\rm Desing}(\pi)=D$ is
\begin{equation}
\label{eq:master}
  N_{n,b}(D)
  \;=\; \sum_{t} F(D,t)\cdot \SGE(n-t,\,b-t),
\end{equation}
where $F(D,t)$ is given by Lemma~\ref{lem:FGEN}. 
The sum runs over all $t$ permitted by Equation (\ref{limit on num of singletons}).
\end{prop}

\begin{example}
Let $n=12$ and $D=\{1,4,8\}$, then
the gap-lengths are: $L_0=0,\break L_1=1,\ L_2=2$ and $L_3=3$. So:
\begin{eqnarray*}
\sum_{t} F(\{1,4,8\},t)\,x^{\,t-3} &=& 1 \cdot (1+x) \cdot (1+x+x^2) \cdot \bigl(1+x+x^2+x^3\bigr)=\\
&=&
1+3x+5x^2+6x^3+5x^4+3x^5+x^6,
\end{eqnarray*}
and hence:
$$F(\{1,4,8\},t)=\left\{ \begin{array}{cl}
    1 & \ t=3,9 \\
    3 & \ t=4,8\\
    5 & \ t=5,7\\
    6 & \ t=6
\end{array}\right.
$$
If the number of blocks is $b=8$, then only $3\leq t \leq 7$ contribute, and we get:
\begin{eqnarray*}
    N_{12,8}(\{1,4,8\})&=&
    F(\{1,4,8\},3)\cdot S_{\geq2} (9,5)+F(\{1,4,8\},4)\cdot S_{\geq2} (8,4)+\\
    && \ +\ F(\{1,4,8\},5)\cdot S_{\geq2} (7,3)+F(\{1,4,8\},6)\cdot S_{\geq2} (6,2)+\\&&
    \ + \ F(\{1,4,8\},7)\cdot S_{\geq2} (5,1)=\\
&=&1 \cdot 0 + 3 \cdot 105 + 5 \cdot 105+6 \cdot 25 + 5 \cdot 1
  = 995,
\end{eqnarray*}
where the values of  $S_{\geq2} (n,k)$ can be found in Table \ref{table A008299} above.
\end{example}

\section{Open questions}\label{section: open questions}
In various parts of this paper, we have raised several questions and we summarize them here, with an appropriate reference to the corresponding context:
\begin{enumerate}
\item {\bf Question \ref{question generalize Eu}:}
Eu \cite{Eu} provides a simple recursive bijection between Motzkin paths of length $n$ and ordinary standard Young tableaux having
$n$ elements with at most three rows. Can Eu's bijection be adapted for proving the equality between the number of non-crossing partitions $\pi$ with ${\rm desing}(\pi)=0$ and the number of shifted Young tableaux having $n$ elements with at most $3$ rows?

\medskip

\item {\bf Question \ref{question equality columns enk}:} We have shown in Claim  \ref{claim_columns_enk} that for $k\leq n-b$ or $k> 2(n-b)$, we have $e_{n-k,k}^{(b)} = e_{n-k,k+1}^{(b+1)}$ using algebraic manipulations, where $e_{n-k,k}^{(b)}$ are the coefficients in the Schur expansion for set partitions $\pi$ of $[n-k]$ having $b$ blocks and ${\rm desing}(\pi)=k$ . Can one find a bijection which justifies this equality?

\medskip

\item {\bf Questions \ref{question comb connection 2nd eulerian to stirling} and \ref{question comb connection 2nd eulerian to en0b}:} Find a combinatorial justification for the following equalities, which were obtained by algebraic manipulations:
\begin{enumerate}
    \item $\sum\limits_{j=0}^{n-b} {{n-b-j} \choose{n-2b+1}}E_2(n-b, j)=S_{\geq2}(n,b)+S_{\geq2}(n-1,b-1)$,
    \item $e_{n,0}^{(b)}=\left\{
\begin{array}{ll}
\sum\limits_{i=0}^{\left\lfloor\frac{b-1}{2}\right\rfloor} \sum\limits_{j=0}^{n-b} {{n-b-j} \choose  {n-2b+2i+1}}E_2(n-b, j) & \ \ n \neq 2b \medskip \\
\sum\limits_{i=0}^{\left\lfloor\frac{b}{2}\right\rfloor} \sum\limits_{j=0}^{b} {{b-j} \choose  {2i}}E_2(b, j) & \ \  n=2b \\
\end{array}\right.$,
\end{enumerate}
where $E_2(b, j)$ are the second-order Eulerian numbers and $S_{\geq2}(n,b)$ are the associated Stirling numbers of the second kind.

\end{enumerate}

An additional question that can be raised at this concluding point, is the following:
\begin{question}
Can the Schur-positivity results of this paper be lifted to a
dual-equivalence theory for set partitions?

More precisely, let \(\mathcal{A}\) be one of the families of set partitions considered
above, equipped with either the statistic \(\operatorname{Short}\) or
\(\operatorname{Desing}\).  Is it possible to define local involutions
\[
\varphi_i:\mathcal{A}\to \mathcal{A},\qquad 2\le i\le n-1,
\]
depending only on the local configuration of the elements
\(i-1,i,i+1\), such that the resulting colored graph is a dual-equivalence
graph, in the sense of Assaf \cite{Assaf}?  Equivalently, can the connected components be made
descent-isomorphic to the standard dual-equivalence graphs on
\(SYT(\lambda)\), where the shapes \(\lambda\) are the hook shapes
\((n-k,1^k)\) in the case of \(\operatorname{Short}\), and the two-row
shapes \((n-k,k)\) in the case of \(\operatorname{Desing}\)?
\end{question}

\medskip

\section*{Acknowledgments}
We would like to express our gratitude to Yuval Roichman for fruitful discussions.

\end{document}